\documentclass[10pt]{amsart} \usepackage[dvips]{epsfig}
\usepackage{amsmath,amssymb} \usepackage[all, knot]{xy}
\usepackage{youngtab} \SelectTips{cm}{} \xyoption{arc}
\xyoption{2cell} \UseAllTwocells
\begin{document}

\setcounter{secnumdepth}{2}
\setcounter{tocdepth}{2}


\newcommand \driicell [1] {\drtwocell \omit {#1}}
\newcommand \ddriicell [1] {\ddrtwocell \omit {#1}}
\newcommand \drriicell [1] {\drrtwocell \omit {#1}}
\newcommand \uriicell [1] {\urtwocell \omit {#1}}
\newcommand \uuriicell [1] {\uurtwocell \omit {#1}}
\newcommand \urriicell [1] {\urrtwocell \omit {#1}}
\newcommand \urrriicell [1] {\urrrtwocell \omit {#1}}


\newtheorem{theorem}{Theorem}
\newtheorem{lemma}{Lemma}
\newtheorem{corollary}{Corollary}
\newtheorem{varexample}{Example}
\newenvironment{example}{\begin{varexample}\em}{\em\end{varexample}}
\newtheorem{definition}{Definition}
\newtheorem{varremark}{Remark}
\newtheorem{proposition}{Proposition}
\newenvironment{remark}{\begin{varremark}\em}{\em\end{varremark}}


\newcommand{\ra}{\mathop{\rightarrow}}
\newcommand{\ralim}{\mathop{\ra}\limits}
\newcommand{\la}{\mathop{\leftarrow}}
\newcommand{\lalim}{\mathop{\la}\limits}
\newcommand{\Longrightarrowlim}{\mathop{\longrightarrow}\limits}
\newcommand{\longrightarrowlim}{\mathop{\longrightarrow}\limits}
\newcommand{\longleftarrowlim}{\mathop{\longleftarrow}\limits}


\newcommand{\opname}[1]{\operatorname{#1}}
\newcommand{\bldsym}[1]{\boldsymbol{#1}}
\newcommand{\catname}[1]{\boldsymbol{\opname{{#1}}}}
\newcommand{\ul}[1]{\underline{#1}}

\newcommand{\inprod}[1]{\left\langle{#1}\right\rangle}
\newcommand{\wquot}{/\!\!/}


\newcommand{\C}{\catname{C}}
\newcommand{\Cop}{\catname{C}^{\opname{op}}}

\newcommand{\Cat}{\catname{Cat}}
\newcommand{\Gpd}{\catname{Gpd}}
\newcommand{\Set}{\catname{Set}}
\newcommand{\Bicat}{\catname{Bicat}}

\newcommand{\Hor}{\catname{Hor}}
\newcommand{\Ver}{\catname{Ver}}
\newcommand{\Squ}{\catname{Squ}}

\newcommand{\iiCob}{\catname{2Cob}}
\newcommand{\nCobi}{\catname{nCob}}
\newcommand{\nCob}{\catname{nCob}_2}
\newcommand{\iiiCob}{\catname{3Cob}_2}

\newcommand{\Span}{\opname{Span}}
\newcommand{\Cspan}{\opname{Span}(\catname{C})}
\newcommand{\Cosp}{\opname{Cosp}}
\newcommand{\CCosp}{\opname{Cosp}(\catname{C})}
\newcommand{\iiCCosp}{\opname{2Cosp}(\catname{C})}

\newcommand{\uispgpd}{\Span(\Gpd)^{U(1)}}

\newcommand{\Ob}{\catname{Obj}}
\newcommand{\M}{\catname{Mor}}
\newcommand{\B}{\catname{2Mor}}

\newcommand{\V}{\catname{Vect}}
\newcommand{\iiV}{\catname{2Vect}}
\newcommand{\Hilb}{\catname{Hilb}}
\newcommand{\iiH}{\catname{2Hilb}}
\newcommand{\ManCorn}{\catname{ManCorn}}

\newcommand{\CG}{\mathbb{C}[G]}
\newcommand{\ZCG}{Z(\mathbb{C}[G])}
\newcommand{\VG}{\catname{Vect}[G]}
\newcommand{\ZVG}{Z(\catname{Vect}[G])}
\newcommand{\ZVS}{Z(\catname{Vect}[SU(2)])}


\newcommand{\Obj}{\opname{Obj}}
\newcommand{\Mor}{\opname{Mor}}
\newcommand{\Path}{\opname{Path}}

\newcommand{\id}{\opname{id}}
\newcommand{\mathd}{\mathrm{d}}
\newcommand{\mathi}{\mathrm{i}}
\newcommand{\mathe}{\mathrm{e}}
\newcommand{\br}[1]{\langle {#1} \rangle}
\newcommand{\Id}{\opname{Id}}

\newcommand{\A}{\mathcal{A}}
\newcommand{\G}{\mathcal{G}}
\newcommand{\fc}[1]{\mathcal{A}_0({#1})_G}
\newcommand{\Z}[1]{\bigl{[} \fc{#1} , \V \bigr{]}}
\newcommand{\HV}[1]{[#1,\V]}
\newcommand{\FV}{\Lambda}


\title{Cohomological Twisting of 2-Linearization and Extended TQFT}
\author{Jeffrey C. Morton}
\address{Dr. Jeffrey Morton, Universität Hamburg, Fakultät für Mathematik, Informatik und Naturwissenschaften, Fachbereich Mathematik, Bundesstraße 55, 20146 Hamburg}
\email{\tt{Jeffrey.Morton@uni-hamburg.de}}

\begin{abstract}
  \addcontentsline{toc}{chapter}{Abstract} In this paper, we describe
  a relation between a categorical quantization construction, called
  ``2-linearization'', and extended topological quantum field theory
  (ETQFT). We then describe an extension of the 2-linearization
  process which incorporates cohomological twisting.  The
  2-linearization process assigns 2-vector spaces to (finite)
  groupoids, functors between them to spans of groupoids, and natural
  transformations to spans between these.  By applying this to
  groupoids which represent the (discrete) moduli spaces for
  topological gauge theory with finite group $G$, the ETQFT obtained
  is the untwisted Dijkgraaf-Witten (DW) model associated to $G$. This
  illustrates the factorization of TQFT into ``classical field
  theory'' valued in groupoids, and ``quantization functors'', which
  has been described by Freed, Hopkins, Lurie and Teleman.  We then
  describe how to extend this to the full DW model, by using a
  generalization of the symmetric monoidal bicategory of groupoids and
  spans which incorporates cocycles.  We give a generalization of the
  2-linearization functor which acts on groupoids and spans which have
  associated cohomological data.  We show how the 3-cocycle $\omega$
  on the classifying space $BG$ which appears in the action for the DW
  model induces a classical field theory valued in this bicategory.
\end{abstract}
\maketitle

\section{Introduction}

This paper demonstrates a construction of an extended topological
quantum field theory (ETQFT), associated to any finite group $G$,
together with an element of the group cohomology, $\omega \in
H^3_{grp}(G,U(1)$. The ETQFT $Z_G$ defined here is related to a
topological gauge theory with finite gauge group $G$, and in
particular to the Dijkgraaf-Witten (DW) model based on $(G,\omega)$.

This construction involves two parts. The first part is a 2-functor
$\FV$, called ``2-linearization'', described by the author in
\cite{gpd2vect}. At the level of objects, the 2-functor $\FV$ takes
finite groupoids to 2-vector spaces.  A ``twisted'' form of $\FV$ is
introduced here, which extends this to a larger 2-category in which
the groupoids carry some cohomological data. The second part is a
functor which, to any space, assigns the groupoid of flat $G$-bundles
with connection. The cohomological data is derived from $\omega$.

There are three main purposes for this paper. First, we show how $\FV$
gives a categorical framework in which to understand some well-known
constructions. Second, we see how these constructions give a
physically-motivated interpretation of $\FV$. Third, we draw on this
motivation to see how to generalize $\FV$ to the more physically
meaningful ``twisted'' case.

Specifically, we first see how $\FV$ can be used to reconstruct the
untwisted case of the 3D TQFT constructed originally by Dijkgraaf and
Witten \cite{DW}, and further developed by Freed and Quinn
\cite{freed-quinn}, as well as elsewhere.  The characterization $\FV$
in terms of ambi-adjoint functors in \cite{gpd2vect} gives a
conceptual account of various normalization factors which appear in
the standard construction, and organizes the physical structure into a
very general structure. In particular, it shows that the DW model
factors into two parts: first, a ``classical field theory'' valued in
a bicategory $\Span(\Gpd)$ whose objects are groupoids and whose
morphisms are ``span'' diagrams; second, a ``quantization functor''
$\FV$.

We are led to a physical interpretation of the 2-linearization functor
$\FV$ as a categorification of the path integrals (in the discrete
case, ``sums over histories'') which are used in the construction of
the DW model. For this reason, and following the terminology of Freed,
Hopkins, Lurie, and Teleman \cite{FHLT}, we often also refer to it as
a ``quantization functor''.

The application to this model leads naturally to a generalization of
$\FV$. Although $\FV$ is a canonical choice for quantization of
$\Span(\Gpd)$, it is unable to express even full generality of the DW
model, itself toy model of a quantum field theory.  In particular, the
general DW model is ``twisted'' by the group cocycle $\omega$
mentioned above, which gives a nontrivial topological action term.
The cocycle is supplied by the classical part of the twisted theory
and can be described in the language of bundles and gerbes
\cite{willerton}.  The cocycle leads to a Lagrangian ``action
functional'', which is a function valued in $U(1)$, and is contant in the
untwisted case.

Extending to this case leads to a larger category $\uispgpd$ of
groupoids and spans equipped with cohomological data, and to a
modification of the quantization functor, $\FV^{U(1)}$, which acts on
this larger category.

This is an example of the program of Freed \cite{freed} on the use of
higher-algebraic structures for quantization, further developed in the
elegant higher-algebraic framework by Freed, Hopkins, Lurie and
Teleman \cite{FHLT}, in which there are two ingredients to an extended
TQFT:
\begin{enumerate} 
\item A "classical" field theory, where the values of the fields live
  in an $n$-groupoid.
\item A "quantization functor" which takes the $n$-groupoids to
  $(n+1)$-algebras and spans to morphisms of all degrees.
\end{enumerate}
That work refers to ``canonical quantization for classical field
theories valued in $n$-groupoids''.  We do not follow this in full
generality - in particular, to reproduce the DW model,
we need consider only the situation down to codimension 2, unlike the
full program of \cite{FHLT}. In particular, our classical field theory
takes values in ordinary groupoids (rather than the more general
setting of $n$-groupoids), equipped with some extra data.

The main result of this paper is a construction of an extended
Topological Quantum Field Theory as a symmetric monoidal 2-functor
from a certain bicategory of double cobordisms into 2-vector spaces,
which reproduces the twisted DW model, in the general case of surfaces
with boundary, as described by Freed and Quinn \cite{freed-quinn}.
The key ingredients are $\fc{-}^{\omega}$, given in Definition
\ref{def:twistedconn}, and $\FV^{U(1)}$, given in Definition
\ref{def:twistedlambda}). The main theorem is:

\textbf{Theorem \ref{thm:maintheorem}}: Given a finite gauge group $G$
and 3-cocycle $\omega \in Z^3(BG,U(1))$, the symmetric monoidal
2-functor
\begin{equation}
  Z_G^{\omega} =  \FV^{U(1)} \circ \fc{-}^{\omega} : \catname{3Cob_2} \ra \iiV
\end{equation}
reproduces the DW model with twisting cocycle $\omega$.

This factorization raises the possibility of applying the same
quantization functor to other examples. The untwisted $\FV$ is a
natural extension of the Baez-Dolan program of ``groupoidification''
(see, e.g. \cite{bhw-gpd, bhw, gpd2vect}), in which
the category $\Span(\Gpd)$ appears in many different contexts. The
twisted variant $\FV^{U(1)}$ extends this further, and we expect that
it should give generalizations of many such examples.

The plan of the paper is as follows: in Section \ref{sec:TQFT} we
recall the categorical setup of extended topological quantum field
theories; in Section \ref{sec:etqft} we describe the untwisted form of
the construction using 2-linearization as the quantization functor and
show that it reproduces the untwisted DW model; in
Section \ref{sec:examples} we give some calculations to explicitly
show some of the invariants computed by this process; in Section
\ref{sec:twisting} we describe the cocycle-twisted variations of the
gauge theory and the 2-linearization functor and show that these
reproduce the twisted DW model; finally we offer some
concluding remarks.

\section{Topological Quantum Field Theories}\label{sec:TQFT}

Here we summarize the context of TQFT and ETQFT in which we will be
working, in categorical language.  We will assume in the following a
basic familiarity with 2-categories, and refer the reader to works
such as those by Cheng and Lauda \cite{chenglauda}, or Lack
\cite{lack-2cat}. For a good introduction on higher category theory in
the context of TQFTs, see work of Baez and Dolan \cite{hdatqft}.

We recall some background on the DW model, to give more context
for the remarks above. Atiyah's axioms for an $n$-dimensional TQFT
describe it as a symmetric monoidal functor
\begin{equation}
  Z : \nCobi \ra \Hilb\
\end{equation} 
where $\nCobi$ is a category whose objects are $(n-1)$-manifolds and
whose morphisms are cobordisms.  In general, a ``$k$-tuply extended
TQFT'' assigns higher-categorical structures called $k$-vector spaces
to manifolds of codimension $k$.  In particular, it is a (weak,
monoidal) $k$-functor:
\begin{equation}
Z : \catname{nCob_k} \ra \catname{k-Vect}
\end{equation}
(The relevance of 2-vector spaces to the setting of topology as in
ETQFT has been described in more detail, for example, by Yetter
\cite{yet}.)  In this paper we are only interested in the case $k =
2$, though the construction given here might be generalized to higher
$k$.

A topological quantum field theory (TQFT) is understood physically as
a quantum field theory with \textit{no local degrees of freedom}.  In
particular, we are interested in TQFTs given by gauge theories.
Fields in gauge theory are connections on bundles over some base
space. We assume such connections are invariant under one-parameter
families of diffeomorphisms, that is, are flat. Thus, the only
interesting information about them is given by holonomies around
non-contractible loops.

\subsection{The Category $\nCobi$}\label{sec:ncob}

In general, a TQFT can be described as a functor from a category of
manifolds and cobordisms into vector spaces (or Hilbert spaces) and
linear maps.  A survey of categorical aspects of TQFT was given by
Bartlett \cite{bartlett}.  For our purposes, we first need to
understand the category of cobordisms involved here.

A \textit{cobordism} between compact manifolds $S_1, S_2$ is a compact
manifold with boundary, $M$, with $\partial M$ isomorphic to the
disjoint union $S_1 \coprod S_2$.  Cobordisms are composed by
identifying their boundaries. For our purposes, it will be useful to
think of cobordisms as special kinds of \textit{cospans} in a category
of manifolds with boundary, so that this is a special case of the
composition of cospans by pushout.  Our aim here is to describe a
generalization of categories of cobordisms.

The bicategory $\nCob$ has:
\begin{itemize}
\item \textbf{Objects} Compact $(n-2)$-manifolds $P$ (without boundary)
\item \textbf{Morphisms}: cobordisms $P_1 \times I \ralim^{i_1} S
  \lalim^{i_2} P_2 \times I$ where $S$ is an $(n-1)$-dimensional
  collared cobordism
\item The 2-morphisms of $\nCob$ are generated by $n$-dimensional
  collared cobordisms with corners $M$, up to diffeomorphisms which
  preserve the boundary
\end{itemize} 
Composition is by gluing along collars, which are included to
ensure that there is a smooth structure on composites.

\begin{remark}
  In the following, all manifolds are compact, whether or not this is
  explicitly mentioned.
\end{remark}

In \cite{dblbicat} there is a definition of a bicategory $\nCob$, as a
sub-bicategory of $\Cosp^2(\ManCorn)$, where $\ManCorn$ is the
category of manifolds with corners.  In fact, the definition given in
\cite{dblbicat} (as a ``double bicategory'') is slightly trickier than
what we will use to define our TQFT. In particular, it is a
\textit{cubical} 2-category, which distinguishes between different
classes 2-morphisms. Specifically, cobordisms of cobordisms are
represented as squares, which we think of as describing ``evolution of
manifolds with boundary'' in which the boundary need not be fixed. If
the source and target are the horizontal edges of a square, the
changing boundary would then be represented by $(n-1)$-dimensional
cobordisms thought of as the vertical edges. Collar-fixing
diffeomorphisms of cobordisms, on the other hand, are ``bigons'' with
a single source and target 1-morphism. When we collapse to a
bicategory, this distinction can be ignored, since any diffeomorphism
has a corresponding ``mapping cylinder'' cobordism.

For our purposes, it is enough to use, as shown in \cite{dblbicat},
the bicategory obtained from the double bicategory described there,
which is equivalent to the usual definition of $\nCob$.  However, we
want to emphasize here that treating the cobordism category explicitly
as a category of cospans means that the ``classical field theory''
functor which we construct is really just induced by a functor on
$\ManCorn$. Specifically, it is induced by the contravariant functor
$\fc{-}$, which is simple to construct, and when applied to cospans of
manifolds with corners, is easily seen to give spans of groupoids.

Details about smooth structure can be largely passed over here.
Indeed, for our construction to work, it is only necessary to get an
$n$-functor from the cobordism category into the appropriate form of
$\Span(\Gpd)$.  Since this is done, here by passing through the
fundamental groupoid, only the homotopy types of the manifolds and
cobordisms are detected by these invariants.  Thus, the precise
details of composing cobordisms with collars is not crucial for these
ETQFTs based on gauge theory.  It may be relevant for other field
theories, however.  TQFT is a special situation, which we now recall.

\subsection{TQFT and ETQFT as Functors}\label{sec:tqftfunctor}

Atiyah's axiomatic formulation \cite{atiyah} of the axioms for a TQFT
can be summarized as follows:
\begin{definition}
 An $n$-dimensional Topological Quantum Field Theory is a
(symmetric) monoidal functor
\begin{equation}
Z: \catname{nCob} \ra \V
\end{equation}
where $\catname{nCob}$ is the monoidal category of $(n-1)$-dimensional
manifolds and $n$-dimensional cobordisms, and $\V$ is the monoidal
category whose objects are vector spaces, whose arrows are linear
transformations, and whose monoidal operation is the usual tensor
product $\otimes$.
\end{definition}

A more general characterization of cobordism categories is the
Baez-Dolan Cobordism Hypothesis, characterizing the $n$-category whose
objects are points and whose $n$-morphisms are $n$-dimensional
cobordisms (necessarily with corners).  The characterization is that
this category is a free symmetric monoidal $n$-category ``with duals''
in a suitable sense (details can be found in \cite{hdatqft}).  This
has been proved by Lurie (see \cite{lurie-cobcat}).  Christopher
Schommer-Pries has given a presentation for $\catname{2Cob_2}$ as a
symmetric monoidal bicategory \cite{csp-class-2d}, given in terms of Morse
theory and a classification of singularities.

This takes us to ETQFTs, which are defined not just on manifolds with
boundary, but also on manifolds with corners. This idea was introduced
by Ruth Lawrence \cite{lawrence}, under the name ``$r$-ETFT'',
replacing the concept of vector space with that of $r$-vector space,
Just as a TQFT assigns a space of states to a manifold and a linear
map to a cobordisms, a (doubly) extended TQFT will assign some such
algebraic data to manifolds of dimension $(n-r)$, and cobordisms up to
dimension $n$. Our construction here will describe the situation where
$r = 2$.

\begin{definition}
The 2-category $\iiV$ has: 
\begin{itemize}
\item \textbf{Objects}: Finite-dimensional Kapranov-Voevodsky 2-vector
  spaces (i.e. $\mathbb{C}$-linear, finitely semisimple abelian
  categories)
\item \textbf{Morphisms}: 2-linear maps (i.e. $\mathbb{C}$-linear
  functors, which are necessarily additive)
\item \textbf{2-Morphisms}: Natural transformations
\end{itemize}
\end{definition}
This is a slight abstraction of the definition given in \cite{KV}.
This category has a natural symmetric monoidal structure, using the
Deligne tensor product of categories.

A straightforward categorification of Atiyah's description of
a TQFT as a functor, as proposed by Lawrence, runs as follows:
\begin{definition}An \textbf{extended TQFT} is a (symmetric monoidal)
  weak 2-functor
  \begin{equation}
    Z : \nCob \ra \iiV
  \end{equation}
\end{definition}

In particular, such a $Z$ assigns:
\begin{itemize}
\item To an $(n-2)$-manifold, a 2-vector space
\item To an $(n-1)$-manifold, a 2-linear map between 2-vector spaces
\item To an $n$-manifold, a 2-natural transformation between 2-linear
  maps
\end{itemize} All this data satisfies the conditions for a weak
2-functor (e.g. it preserves composition and units up to coherent
isomorphism, and so forth). The monoidal structure in $\iiV$ is the
Deligne tensor product on abelian categories (see e.g. section 4.3 of
\cite{KL}).

As before, a fuller version of this theory will use $\catname{2Hilb}$
(see \cite{hdaii}) in place of $\iiV$, but we will mostly omit this
complication here.

\subsection{Topological Gauge Theory and
  TQFTs}\label{sec:topgaugetheory}

Quantum field theories are often derived from classical field theories
described in terms of gauge theory.  In the topological case, a class
of TQFTs and ETQFTs may be constructed from topological gauge theory,
in which fields are flat connections on a manifold $B$. Being flat,
the nontrivial information about a connection depends only on the
topology of $B$.

In particular, all the information available about any connection
comes in the form of holonomies of the connection around loops. The
holonomy is an element $A(\gamma) \in G$ associated to a loop $\gamma$
in $B$, defining the ``parallel transport'' around that loop. The
$G$-connection is flat exactly if the holonomy assigned to a loop
depends only on the homotopy class of $\gamma$.
To say this more conveniently, we first recall the definition (see
Brown \cite{brown}):

\begin{definition}
The \textbf{fundamental groupoid} $\Pi_1(B)$ of a space $B$ is a groupoid
with points of $B$ as its objects, and whose morphisms from $x$ to $y$
are just all homotopy classes of paths in $B$ starting at $x$ and ending
at $y$.
\end{definition}

Suppose $G$ is a group, understood as a one-object groupoid whose
composition is the group multiplication.  Then we have:

\begin{definition}\label{def:gpdconn}
  A \textbf{flat $G$-connection} is a functor
  \begin{equation}
    A : \Pi_1(B) \ra G
  \end{equation}
  A \textbf{gauge transformation} $\alpha : A \ra A'$ from one
  connection to another is a natural transformation of functors so
  that $\alpha_x \in G$ satisfies $\alpha_y A(\gamma) = A'(\gamma)
  \alpha_x$ for each path $\gamma : x \ra y$.  Flat connections and
  natural transformations form the objects and morphisms of the
  \textit{groupoid of flat connections}
  \begin{equation}
    \fc{B} = [ \Pi_1(B) , G ]
  \end{equation}
  The \textbf{groupoid of flat connections functor} is the contravariant functor:
  \begin{equation}
    \fc{-} : \ManCorn \ra \Gpd
  \end{equation}
  which, to $(X \ralim^{f} Y) \in \ManCorn$, assigns the restriction
  map $f^*$, precomposition with $f$.
\end{definition}
(Here we are using the notation that $[ C_1, C_2 ]$ is the category
whose objects are functors from $C_1$ to $C_2$ and whose morphisms are
natural transformations.)

In Section \ref{sec:cobspan} we will see how this extends to a functor
on $Cosp^2(\ManCorn)$, and therefore to $\nCob$, which takes values in
a bicategory of groupoids and spans.

Thurston \cite{thurston} showed that this groupoid of connections is
equivalent, in the categorical sense, to the usual definition in terms
of flat principal $G$-bundles. This is because, as categories, $G
\simeq \catname{G-Tor}$, the category of $G$-torsors (sets with a
$G$-action which are isomorphic as $G$-sets to $G$ itself). A flat
connection on a principal $G$-bundle gives the fiber-selecting functor
from $\Pi_1(B)$ to $\catname{G-Tor}$, where the holonomy along a path
transports each fiber. It is enough to use the groupoid of
connections, since we are only interested in these constructions up to
equivalence, so any representative of the equivalence class.

(We note that, more precisely, the configuration spaces should be seen
as stacks, which are determined by Morita equivalence classes of
groupoids, and consequently everything we construct here is unchanged,
up to equivalence, by taking Morita equivalent groupoids everywhere.
This refinement is important for topological groupoids, but here we
need not be concerned with it, since Morita equivalence and
categorical equivalence are the same for finite groupoids.)

In gauge theory, two connections which are related by a gauge
transformation describe physically indistinguishable states - the
differences they detect are due only to the system of measurement
used. The groupoid then describes the internal symmetry of a
``physical'' space of states.  Now, rewriting definition
\ref{def:gpdconn}, if $\gamma : x \ra x$ in $\Pi_1(B)$ is a loop, and
$A$ and $A'$ are two connections related by a gauge transformation
$\alpha$, the relation between $A$ and $A'$ can be expressed
$A'(\gamma) = \alpha(x)^{-1} A(\gamma) \alpha(x)$, that is, the
holonomies assigned by the two connections around the loop are
conjugate. So physically distinct holonomies correspond to conjugacy
classes in $G$.

Indeed any category is equivalent, as a category, to its skeleton, so
in general $\Pi_1(B) \cong \coprod_{b \in \pi_0(B)} \pi_1(B,b)$. The
gauge transformations for connections measured from a fixed base point
$b$ are determined by a single group element at $b$, acting on
holonomies around any loop by conjugation.  The groupoid $\fc{B}$,
which is the configuration space for our theory, is the ``weak
quotient'' of the space $Fun(\pi_1(B),G)$ of connections by the action
of gauge transformations at each base-point.

\begin{proposition}\label{thm:ZB2VS} 
  For any compact manifold $B$ (possibly with corners), and finite
  group $G$, the groupoid $\fc{B}$ is essentially finite (equivalent
  as a category to a finite groupoid).
\end{proposition}\label{thm:fcBfinite}
\begin{proof}
  To begin with, note that the functor category $\fc{B}$ is indeed a
  groupoid, since any natural transformation $g$ assigns to a point $b
  \in B$ a group element $g_b$, which is invertible. The
  transformation $g^{-1}$ with $g^{-1}_b = (g_b)^{-1}$ is the inverse.

  Next, note that for any space $B$,
  \begin{equation}
    \Pi_1(B) \equiv \coprod_{i=1}^{n}(\pi_1(B_i))
  \end{equation}
  The sum is taken over all path components of $B$. That is, objects
  in $\Pi_1(B)$ are by definition isomorphic if and only if they are
  in the same path component.  The automorphisms for an object
  corresponding to path component $B_i$ are the equivalence classes of
  paths from any chosen point to itself---namely, $\pi_1(B_i)$.  If
  $B$ is a compact manifold, so is each component $B_i$, which is also
  connected. But the fundamental group for a compact, connected
  manifold is finitely generated. So in particular, each $\pi_1(B_i)$
  is finitely generated, and there are a finite number of
  components. So $\Pi_1(B)$ is an essentially finitely generated
  groupoid.

  But if $\Pi_1(B)$ is essentially finitely generated, then since $G$
  is a finite group, $\fc{B}$ is an essentially finite groupoid. This
  is because each functor's object map is determined by the images of
  the generators, and there are finitely many such
  assignments. Similarly, $\Pi_1(B)$ is equivalent to its skeleton,
  and a natural transformation in this case is just given by a group
  element in $G$ for each component of $B$, so there are finitely
  many.
\end{proof}

We have described how to associated the groupoid $\fc{B}$ to any
manifold $B$.  Next we will see how to use this to construct extended
topological quantum field theories.

\section{ETQFT by 2-Linearization}\label{sec:etqft}

Here we want to consider an explicit construction of a class of
extended TQFTs based on a finite group $G$, using the groupoids of
connections  described in the previous section.

In \cite{gpd2vect}, we defined a weak 2-functor $\FV$ from spans of
groupoids to 2-vector spaces.  In particular, the construction we give
here works by associating spans of groupoids to cobordisms, and then
applying this $\FV$. These groupoids arise from the moduli space
of flat connections on the source and target manifolds, as well as on
the cobordism itself. These are connected by natural restriction maps
to give spans.  A similar line of reasoning gives spans of span maps
associated to cobordisms between cobordisms.

We recap the key ingredient $\FV$ next.

\subsection{Groupoidification and 2-Linearization}

\textit{Groupoidification} is a term for the (non-systematic) process
of finding an inverse to the (systematic) ``degroupoidification''
functor, which gives representations of $\Span(\Gpd)$ in $\V$, or
$\Hilb$. The goal is to find structures in $Span(Gpd)$ whose
representations reproduce some chosen structure in $\V$ or $\Hilb$.
The reader may find more details on this program in a review by Baez,
Hoffnung and Walker \cite{bhw-gpd}, and in Hoffnung's work on
geometric representation theory \cite{hoffnung-hecke}.  The author has
described an example of an application to physics, and in particular
the combinatorics of Feynman diagrams in \cite{catalgQM}.

\textit{2-Linearization}: the 2-functor $\Lambda$ gives a
representation of $\Span(\Gpd)$ in 2-vector spaces. This is discussed
in the general setting in \cite{gpd2vect}. 

Both invariants rely on different forms of a ubiquitous
\textbf{pull-push} process, the best-known example of which is
ordinary matrix multiplication. Indeed the relation to matrix
multiplication is exactly the basis for the degroupoidification
functor.  In the context of the 2-linearization functor $\FV$, the
``pull'' and ``push'' refer to the direct and inverse limits of
$\V$-presheaves along a functor. This fundamental construction appears
in many categories, notably toposes \cite{macmoer}. For abelian
sheaves this is described in some generality by Kashiwara and Schapira
\cite{kasch}. The situation most closely relevant to the present case
occurs in the setting of representations of rings \cite{benson}.

For our immediate purposes, we can omit many of these considerations,
but note that the ambidextrous (i.e. both left and right) adjunction
between direct and inverse image functors valued in $\V$ gives us the
extra structure used to construct $\FV$. This ambidextrous adjunction
appears, indirectly, because a finite-dimensional vector space is
canonically isomorphic to its double-dual. (For this reason, in
infinite-dimensional situations, one properly ought to use
$\Hilb$-valued functors, which may be seen as equivariant Hilbert
bundles.  For the finite case, $\V$ is sufficient.)

\subsection{The 2-Linearization Functor $\FV$}\label{sec:iilin}

The category $Rep(X)$ of finite dimensional representations of a
groupoid $X$ is the category of $X$-actions on bundles of vector
spaces over the objects of $X$. In the essentially finite, discrete
context we are working in, this just the same as the category of
functors from $X$ into $\V$, denoted $[X,\V]$. Such a representation
category is a 2-vector space (i.e. a $\V$-enriched abelian category).

The construction of $\FV$ uses the fact that, for any functor $f : X
\rightarrow Y$ of essentially finite groupoids, there is an adjunction
\begin{equation}
\xymatrix{
  Rep(X) \ar@<2pt>[r]^{f_{\ast}} & Rep(Y) \ar@<2pt>[l]^{f^{\ast}}
}
\end{equation}

We describe these as $f^{\ast}$ (``pull'') and $f_{\ast}$ (``push'')
between the 2-vector spaces of functors $[X,\V]$ and $[Y,\V]$. In
fact, this adjunction is ``ambidextrous'': $f_{\ast}$ is both a left
and a right adjoint to $f^{\ast}$. The importance of ambidextrous
adjunctions is discussed by Lauda \cite{laudaambidjunction} from the
algebraic point of view which relates 2D TQFTs to Frobenius algebras.

The effect of $\FV$ on
2-morphisms can also be thought of in terms of a ``pull-push''
process, but here we use the unit and counit from the two adjunctions
between $f^{\ast}$ and $f_{\ast}$. In particular, we use the unit
from the adjunction where $f_{\ast}$ is a right adjoint, and the
counit from the adjunction where it is a left adjoint. We denote the
unit:
\begin{equation}
\eta_R : \Id_{[X,\V]} \Longrightarrow f_{\ast} f^{\ast}
\end{equation}
The counit is similarly denoted:
\begin{equation}
\epsilon_L : f_{\ast} f^{\ast} \Longrightarrow \Id_{[X,\V]}
\end{equation}

\begin{remark}
We note here that these two operations are a special case of the
general ``six-operation'' framework \cite{moer-groth}: in algebraic
geometry, for a map $f: X \rightarrow Y$ of varieties (or schemes),
one gets functors $f^{\ast}$, $f_{\ast}$, $f^{!}$ and $f_{!}$ between
categories of sheaves $Sh(X)$ and $Sh(Y)$. This is a special case,
since we take our groupoids to have discrete topology, so all functors
(as presheaves) are sheaves. Furthermore, $f^{\ast}$ has in the
general case a different left adjoint $f_{\ast}$ and right adjoint
$f_{!}$.  However, in this case, the two adjoint pairs of functors
coincide.  This is due, indirectly, to the fact that objects in $\V$
are canonically isomorphic to their double-duals, as can be seen by
the matrix representation of these 2-linear maps.
\end{remark}

One way to summarize the structure we get uses a certain bicategory of
spans of groupoids:

\begin{definition}The symmetric monoidal bicategory $\Span(\Gpd)$ has:
\begin{itemize}
\item \textbf{Objects}: Essentially finite groupoids
\item \textbf{Morphisms}: Spans of groupoids
\item \textbf{2-Morphisms}: Equivalence classes of spans of span maps
\end{itemize}
The monoidal operation for $\Span(\Gpd)$ is determined by the fact
that, on objects, it is the product in $\Gpd$.
\end{definition}

This generalizes a construction of a bicategory whose morphisms are
spans, and whose 2-morphisms are span maps.  In fact, $\Span(\Gpd)$ as
we have presented it might be better understood a 3-category.  In
general, the span construction on any bicategory will yield a
(monoidal) tricategory, where the 3-morphisms are maps of spans of
span maps, as described by Hoffnung \cite{hoffnung-span}.  Reducing to
3-isomorphism classes gives exactly the 2-morphisms described here,
and makes our $\Span(\Gpd)$ a (symmetric) monoidal bicategory.  We have chosen the
current approach because of the up-to-diffeomorphism definition of
2-morphisms in $\nCob$.

For a category $\catname{C}$ with pullbacks, the $\Span(\catname{C})$
has many useful properties due to certain universal properties of the
span construction \cite{unispan} (for bicategories, a similar analysis
is done in \cite{kenney-pronk}).  For example, taking categories of
spans ensures that every morphism has a ``dual'' (the same span,
considered with the opposite orientation), and is a minimal expansion
of $\catname{C}$ with this property. The point of the following
construction is to take these ``duals'' and represent them as
ambi-adjoint functors.

Thus, it was shown \cite{gpd2vect} that the following defines a
2-functor:

\begin{definition}
The weak 2-functor
\begin{equation}\label{eq:FVdef}
\FV : \Span(\Gpd) \ra \iiV
\end{equation}
assigns:
\begin{itemize}
\item For $X$ an essentially finite groupoid, the functor category
  $\FV{X} = \HV{X}$
\item For a span of groupoids $A \stackrel{s}{\la} X \stackrel{t}{\ra}
  B$ in $\Span (\Gpd)$, the 2-linear map:
  \begin{equation}\label{eq:FVmor}
    \FV{X} = t_{\ast} \circ s^{\ast} : \FV{A} \longrightarrow \FV{B}
  \end{equation}
\item For a span between spans, $Y : X_1 \ra X_2$ for $X_1,X_2 : A \ra B$, as in:
\begin{equation}\label{eq:FV2mor}
\xymatrix{
 & X_1 \ar[dl]_{s_1} \ar [dr]^{t_1} & \\
A & Y \ar[u]^{s} \ar[d]_{t} & B  \\
 & X_2 \ar[ul]^{s_2} \ar [ur]_{t_2} & \\
}
\end{equation}
the natural transformation
\begin{equation}\label{eq:FV2mor-nakayama}
  \FV(Y) = \epsilon_{L,t} \circ N \circ \eta_{R,s} : (t_1)_{\ast} s_1^{\ast} \Longrightarrow (t_2)_{\ast} s_2^{\ast}
\end{equation}
where $\epsilon_{L,t}$ is the counit for the left adjunction
associated to $t$, and $\eta_{R,s}$ is the unit for the right
adjunction associated to $s$, and $N$ is the Nakayama isomorphism
between the left and right adjoints.
\end{itemize}
\end{definition}

We note that
$\FV$ is a weak 2-functor, so there are also natural isomorphisms
called the ``compositor''
\begin{equation}
 \beta : \FV( X' \circ X ) \ra \FV(X') \circ \FV(X)
\end{equation}
for each composable pair of spans $X$ and $X'$, and the ``unitors''
\begin{equation} 
  U_B : 1_{\FV(B)} \stackrel{\sim}{\ra} \FV(1_B)
\end{equation} for each groupoid $X$. These are described
in \cite{gpd2vect} in detail. So, briefly, is the case where the
2-morphism diagram is only required to commute up to
isomorphism. These issues will not be required in the current context.

The role of the Nakayama isomorphism here is also described in more
detail in \cite{gpd2vect}, but is relevant here, so we will briefly
recap this. In general, given a map $f : X \ra Y$, there will be
both a left and a right adjoint to $f^{\ast}$, the pullback of (in
this case, $\V$-valued) functors from $Y$ to $X$. These may be
described in terms of the internal $\hom$ and $\otimes$ in $\V$.

In each case, these ``pushforwards'' of a functor $F : X \ra \V$ to
$Y$ will be described as a direct sum over all objects $x$ in the
essential preimage of $y \in Y$. Since $F(x)$ gives a representation
of $Aut(x)$, the summands are the induced representations along the
associated homomorphism $\hat{F} : Aut(x) \ra Aut(y)$. For the left
adjoint, this is $\mathbb{C}[Aut(y)] \otimes_{\mathbb{C}[Aut(x)]}F(x)$
(a representation of $Aut(y)$), and for the right adjoint it is
$\hom_{\mathbb{C}[Aut(x)]}(\mathbb{C}[Aut(y)],F(x))$ (that is, the
hom-space as $\mathbb{C}[Aut(x)]$-modules). The Nakayama isomorphism
turns a map in the right adjoint (hom-space) to a vector in the left
adjoint (tensor product) by the ``exterior trace'', averaging over a
group action:
\begin{equation}\label{eq:naka-avg}
  \phi \mapsto \frac{1}{\# Aut(x)} \sum_{g \in Aut(y)} g^{-1} \otimes \phi(g)
\end{equation}

This gives the natural transformations associated to 2-morphisms by
2-linearization. Note that we sum over $Aut(y)$, which projects to the
$Aut(y)$-invariant subspace, a canonical representative of a vector in
the tensor product, but the ``average'' is given by dividing by the
size of $Aut(x)$. This reflects the fact that we are pushing forward a
representation of $Aut(x)$, and it is necessary to make this an
isomorphism when we are dealing with modules in general, say over
$\mathbb{Z}$, rather than $\mathbb{C}$-vector spaces. In this setting,
it merely sets a canonical scale, which turns out to reproduce the
groupoid cardinality which appears in the groupoidification process of
Baez and Dolan (see e.g. \cite{bhw}).

The 2-linearization construction relies on the fact that having both
covariant and contravariant functors $(-)^{\ast}$ and $(-)_{\ast}$
amounts to the same thing as having a single functor from
$\Span(\Gpd)$. In general, pairs of functors like this satisfying some
nice properties are \textit{Mackey functors} (see \cite{elango,
  greenlees}). The situation is in general somewhat more complicated
when groupoids are thought of as having topological spaces, rather
than discrete sets, of objects and morphisms.  However, we take
advantage of the simplifying fact for the discrete case to construct
an ETQFT for a discrete gauge group $G$. We describe this in the next
section.

\subsection{From Cobordisms to Spans}\label{sec:cobspan}

In this section, we show the functoriality of the classical field
theory part of the factorization $Z_G = \FV \circ \fc{-}$.

\begin{theorem}There is a 2-functor:
  \begin{equation}
    \fc{-} : \nCob \ra \Span(\Gpd)
  \end{equation}
  induced by the groupoid of flat connections functor of Definition
  \ref{def:gpdconn}
\end{theorem}
\begin{proof}
  This will follow from the fact that there is an inclusion $\nCob \ra
  \Cosp_2(\ManCorn)$, the bicategory derived from $Cosp^2(\ManCorn)$ as
  in \cite{dblbicat}.

  A cobordism in $\nCob$ can be seen as a particular 1-morphism
  (cospan) in the bicategory $Cosp_2(\ManCorn)$, given by inclusion
  maps:
  \begin{equation}\label{eq:zstep0}
    \xymatrix{
      & S & \\
      B \ar[ur]^{\iota} & & B' \ar[ul]_{\iota'} \\
    }
  \end{equation}
  (Note that as described in \cite{dblbicat}, there are ``collars''
  associated with these inclusions, but these have no effect on our
  construction so we shall ignore them here.)

  Given two cobordisms $S : B_1 \rightarrow B_2$ and $S' : B_2
  \rightarrow B_3$, the composite $S' \circ S$ is a (homotopy) pushout
  of two cospans (over $B_2$). The functor $\Pi_1$ also gives cospans
  of the fundamental groupoids, whose composite $\Pi_1(S') \circ
  \Pi_1(S)$ is a (weak) pushout. Then it is a well-known consequence
  of Brown's \cite{brown} groupoid version of the Van Kampen
  theorem that $\Pi_1(S' \circ S) \simeq \Pi_1(S') \circ \Pi_1(S)$
  (see also \cite{higgins}).

  In the next step, we apply $[ - , G] : Cosp(\Gpd)
  \rightarrow \Span(\Gpd)$. So at this stage of the construction we
  have a span.  To see that this operation is compatible with
  composition of cobordisms the essential fact is that the
  contravariant functor $[-,G]$ takes weak (homotopy) pushouts to weak
  pullbacks.

  The composition of morphisms is by weak pushout of (collared)
  cospans in $\catname{ManCorn}$. This still holds when we take
  fundamental groupoids.  Applying $[-,G]$, we get spans of
  groupoids. Thus the corresponding diagram of spans contains a (weak)
  pullback square.  Denoting the pull-backs along the inclusions by
  $p_i$ and $p'_i$, we have this diagram:

  \begin{equation}\label{xy:conncompos}
    \xymatrix{
      & & \fc{S'} \circ \fc{S} \ar_{P_S}[ld]\ar^{p_{S'}}[rd]\ar@/_2pc/_{p_1 \circ P_S}[ddll]\ar@/^2pc/^{p'_2 \circ P_{S'}}[ddrr] & & \\
      & \fc{S}\ar^{p_1}[ld]\ar_{p_2}[rd] \ar@{=>}[rr]^{\alpha}_{\sim} & & \fc{S'}\ar^{p'_1}[ld]\ar_{p'_2}[rd] & \\
      \fc{B_1} & & \fc{B_2} & & \fc{B_3} \\
    }
  \end{equation}

  The weak pullback $\fc{S'} \circ \fc{S}$ is canonically described
  (up to equivalence) as a comma category, whose objects are triples
  $A,f,A'$ where $A \in \fc{S}$, $A' \in \fc{S'}$, and $f : p_2(A)
  \rightarrow p'_1(A')$.  That is, for connections $A$ and $A'$ on $S$
  and $S'$, the restrictions to $B_2$ are gauge equivalent, but not
  necessarily equal. Each different gauge equivalence gives a
  different object of $\fc{S'} \circ \fc{S}$.

  Thus, $\fc{S'} \circ \fc{S} \simeq \fc{S' \circ S}$, where
  composition of cobordisms is by the (homotopy) pushout along the
  collared inclusions of the boundary $B_2$.

  A similar argument shows the unitor property for $\fc{-}$.

  Finally, for 2-morphisms, we note that here, composition is by
  strict pullback and pushout since spans of spans are taken only up
  to isomorphism. Otherwise the same argument holds. Thus, we have a
  2-functor into $\Span(\Gpd)$.
\end{proof}

\begin{remark}
  We also note here that a similar construction to the functor $[-,G]$
  used here plays a role in a construction by Grandis \cite{grandis3}
  of TQFT via spans of sets. That construction works with topological
  spaces, rather than $\ManCorn$, and $[-,S]$ then denotes the functor
  which assigns to a space $A$ the set of homotopy classes of maps
  from $A$ into $S$. That is, here we are concerned with the homotopy
  1-type (a groupoid), rather than a 0-type (a set) of the space of
  maps into $S$. The result we need is shown in the general case of
  spaces by Chorny \cite{chorny}, and the groupoid case follows since
  groupoids are homotopy 1-types of spaces.
\end{remark}

\subsection{Extended TQFT via $\FV$}

We can now describe explicitly how our ETQFT is constructed:

\begin{definition}
  For any finite group $G$, define the 2-functor
  \begin{equation}
    Z_G = \FV \circ \fc{-} : \nCob \ra \iiV
  \end{equation}
\end{definition}

\begin{proposition}
This $Z_G$ is an Extended TQFT.
\end{proposition}
\begin{proof}
  Since both $\fc{-}$ and $\FV$ are weak, symmetric monoidal
  2-functors, so is the composite $Z_G$, so this is indeed an Extended
  TQFT.
\end{proof}

In particular, since by Proposition \ref{thm:fcBfinite} $\fc{B}$ is an
essentially finite groupoid, the main theorem of \cite{gpd2vect} then
implies $\Z{B}$ is a Kapranov-Voevodsky 2-vector space.

\begin{remark}
  To describe it explicitly, given a finite group $G$, the extended
  TQFT $Z_G$ makes the following assignments:
  \begin{itemize}
  \item For a closed compact manifold $B$, $Z_G$ assigns the 2-vector
    space:
    \begin{equation}\label{eq:ZGonB}
      Z_G(B) = \Z{B}
    \end{equation}
  \item For a cobordism between manifolds:
    \begin{equation}
      B \ralim^{i} S \lalim^{i'} B'
    \end{equation}
    the weak 2-functor assigns a 2-linear map:
    \begin{equation}\label{eq:ZGonS}
      Z_G(S) = (p')_{\ast} \circ p^{\ast}
    \end{equation}
    where $p$ and $p'$ are the projections for the span of
    groupoids associated to $S$ by $\fc{-}$.
  \item For a cobordism with corners between two cobordisms with the
    same source and target:
    \begin{equation}
      \xymatrix{
        & S_1 \ar[d]_{i} & \\
        B \ar[ur]^{i_1} \ar[dr]_{i_2} & M & B' \ar[ul]_{i'_1} \ar[dl]^{i'_2} \\
        & S_2 \ar[u]^{i'} & \\
      }
    \end{equation}
    the natural transformation (\ref{eq:FV2mor-nakayama}) becomes:
    \begin{equation}\label{eq:2morcobfmla}
      \FV(M) = \epsilon_{L,p'} \circ N \circ \eta_{R,p} : (p'_1)_{\ast} \circ {p_1}^{\ast} \Longrightarrow (p'_2)_{\ast} \circ p_2^{\ast}
    \end{equation}
  \end{itemize}
  where $p'$ and $p$ are as above. The coherence isomorphisms which
  make $Z_G$ a weak 2-functor are those defined by $\FV$ as in
  \cite{gpd2vect}. (These appear in coordinates as matrices whose
  components are linear maps between the coefficients of the 2-linear
  maps).
\end{remark}

\subsection{The 3D Untwisted Dijkgraaf-Witten Model as
  ETQFT}\label{sec:DWmodel}

Now we will consider the three-dimensional case in particular. Given a
gauge group $G$, the DW model \cite{DW} is a
topological gauge theory, involving flat $G$-connections on
manifolds. For Lie groups, this theory is related to the Chern-Simons
theory, but our interest here is for finite groups. The general theory
of Lie groups can be understood from finite groups and simply
connected Lie groups. This is because, as described by \cite{DW}, the
finite groups occur in exact sequences as either the group of
components, or the fundamental group, of Lie groups (which can thus be
used to reduce a general Lie group first to a connected, then a
simply-connected, one).

TQFTs equivalent to the DW model are often defined as invariants for
triangulated manifolds, found by considering compatible $G$-colorings
of the (directed) edges. This is done, for example, by Yetter
\cite{yettqft}, and in discussion in the chapter on TQFT of the
unpublished notes by Porter \cite{menagerie}, which also discuss an
extension to categorical groups. Yetter showed that one can obtain an
invariant of manifolds which is independent of triangulation via a
colimit over all triangulations.

While triangulations are crucial in the case of categorical groups,
for ordinary groups, the $G$-colorings of edges in a triangulation
amount to flat $G$-connections. These can be described in terms of
maps in $Hom(\pi_1(M),G)$, or equivalently in $Hom(M,BG)$, where $BG$
is the classifying space of $G$. We use the former description here,
since the groupoid structure is easiest to see in that form.

Note that the normalizing factors which appear in the 2-functor $\FV$
as the Nakayama isomorphism (\ref{eq:naka-avg}) count the size of
automorphism groups of objects. In the mapping space $Hom(M,BG)$,
these appear as the size of homotopy groups of connected components
(as in the ``homotopy order'' as described in \cite{menagerie}). The
homotopy order of a connected space $X$ with base-point $x$ and only
finitely many nontrivial homotopy groups is:
\begin{equation}
  \#^{\pi}(X,x) = \prod_{i=1}^{\infty}|\pi_i(X,x)^{(-1)^i}|
\end{equation}
(In the case of $BG$, and $Hom(M,BG)$, all homotopy groups for $i \geq
2$ are trivial, so this reproduces the groupoid cardinality.)

We now consider explicitly how the DW model and a natural ETQFT
extension of it can be found from the 2-functor $Z_G$ defined in this
paper.

Recall that the category $\iiCob$ occurs in $\iiiCob$ as the category
of automorphisms of the object $\emptyset$, which is particularly
interesting since $\emptyset$ is the monoidal unit in $\iiiCob$. We
can ask about the effect of $Z_G$ restricted to this automorphism
category. It turns out to be just the same as the DW model in 3
dimensions.

Our main theorem is the following:

\begin{theorem}\label{thm:DWETQFT}There is a natural isomorphism
  between the functor $Z_G|_{Aut(\emptyset)}$ and the untwisted DW model.
\end{theorem}
\begin{proof}
We need to exhibit the natural equivalence at the level of objects and
morphisms.

Now, $Z_G(\emptyset) \cong \V$, 
whose single basis 2-vector (mapped to $\mathbb{C}$ under the
equivalence) is the trivial representation of the trivial group.

So a cobordism in $Aut(\emptyset)$ gives a 2-linear map from $\V$ to
$\V$, which is naturally equivalent to giving a vector space (and in
particular, a complex vector space with a specified basis, and thus a
Hilbert space). Cobordisms in $Aut(\emptyset)$ are 2-dimensional
cobordisms from $\emptyset$ to $\emptyset$, or in other words, closed
2-dimensional manifolds. These are, up to diffeomorphism, just
genus-$g$ tori $\Sigma_g$.

Given $\Sigma_g$, the DW model produces a $d$-dimensional Hilbert
space $\mathcal{H}_{\Sigma_g}$, where
\begin{eqnarray}
d & = & |V_g|\\
\nonumber  & = & |\hom(\pi_1(\Sigma_g),G)/G|
\end{eqnarray}
with a basis canonically indexed by conjugacy classes of flat
connections $\gamma \in \hom(\pi_1(\Sigma_g),G)/G$.

Now, thinking of $\Sigma_g$ as a cobordism, that is, as a cospan:
\begin{equation}
\emptyset \ra \Sigma_g \la \emptyset
\end{equation}
we get the span of groupoids
\begin{equation}
\catname{!} \stackrel{!}{\la} \fc{\Sigma_g} \wquot G \stackrel{!}{\ra} \catname{1}
\end{equation}

Here, the map $!$ denotes the unique map into the terminal object,
which we name in order to express the following formulas. By the
above, we find that $Z_G(\Sigma_g) : Z_G(\emptyset) \ra
Z_G(\emptyset)$ can be represented as a $1 \times 1$ matrix, with
\begin{eqnarray}
Z_G(\Sigma_g)_{\mathbb{C},\mathbb{C}} & = & \inprod{!^*(\mathbb{C}),!^*(\mathbb{C})} \\
\nonumber & = & \bigoplus_{\gamma \in \fc{\Sigma_g}/G} \hom(\mathbb{C},\mathbb{C}) \\
\nonumber & \cong & \bigoplus_{\gamma \in \fc{\Sigma_g}/G} \mathbb{C}
\end{eqnarray}
This is canonically isomorphic to $\mathcal{H}_{\Sigma_g}$.

So suppose we have a 3-dimensional cobordism between 2-manifolds
$\Sigma$ and $\Sigma '$, which amounts to a cospan of cospan maps:
\begin{equation}
  \xymatrix{
    & \Sigma \ar[d]& \\
    \emptyset \ar[ru] \ar[r] \ar[rd] & M & \emptyset \ar[lu] \ar[l] \ar[ld] \\
    & \Sigma ' \ar[u] & \\
  }
\end{equation}
then again there is a span of span maps, where the central morphisms
are the restrictions along the boundary inclusion maps:
\begin{equation}
\xymatrix{
 & \fc{\Sigma}  \ar[dl]_{!} \ar [dr]^{!} & \\
1 & \fc{M} \ar[u] \ar[d] & 1  \\
 & \fc{\Sigma '} \ar[ul]^{!} \ar [ur]_{!} & \\
}
\end{equation}
Now, this gives a 2-linear map $Z_G(M)$, and again there is only one
entry, so we find
\begin{equation}
  Z_G(M)_{\mathbb{C},\mathbb{C}} : Z_G(\Sigma)_{\mathbb{C},\mathbb{C}} \ra Z_G(\Sigma')_{\mathbb{C},\mathbb{C}}
\end{equation}

And in particular, we can write the components of this linear map as:
\begin{eqnarray}
  (Z_G(M)_{\mathbb{C},\mathbb{C}})_{\gamma,\gamma '} & = & |(\pi , \pi ')^{-1}([\gamma],[\gamma '])| \\
\nonumber  & = & \sum_{[\gamma ''] \in (\pi , \pi ')^{-1} ( [\gamma], [\gamma '])} \frac{1}{|Aut(\gamma '')|} \\
\nonumber  & = & \sum_{[\gamma '']} \frac{|\{\gamma '' \in [\gamma '']\}|}{|G|} \\
\end{eqnarray}
In this last sum, we have only observed that the size of the
automorphism group of a given connection $\gamma ''$ is just the size
of the full group divided by the size of the orbit of $\gamma ''$,
namely the equivalence class $[\gamma '']$.  But now we can convert a
sum over this equivalence class into a sum over all connections, and
get:
\begin{equation}
 (Z_G(M)_{\mathbb{C},\mathbb{C}})_{\gamma,\gamma '} = \frac{1}{|G|} \sum_{\gamma ''} 1
\end{equation}
So this is the same linear map produced by the DW model.
\end{proof}

The DW model itself is somewhat more general than what we have
discussed so far. In fact, the 2-linearization framework used here
constructs a particular ETQFT, which is the "untwisted" theory.
Twisted DW models may be defined using a class $\alpha$ from the group
cohomology of $G$.  To produce twisted DW models, one must extend the
2-linearization framework to include cocycles. In the twisted
DW model, the ``topological action'' associated to a
given flat connection, is a unit complex number associated to that
connection, determined by $\alpha$.

In Section \ref{sec:twisting} we will describe how to extend this
result through a generalized 2-linearization process.

\section{Example Calculations}\label{sec:examples}

Although the construction for an ETQFT will work in any dimension, its
main features are visible in any dimension at least 2, to allow
codimension-2 submanifolds. Some calculations in low dimensions
illustrate the invariants produced by the ETQFT.

\subsection{$Z_G$ On Manifolds}

We can give the dimension of the 2-vector space assigned to any
manifold by counting its basis objects, which yields the following
straightforward fact:

\begin{proposition}\label{thm:ZBdimension} The KV 2-vector space $Z_G(B)$
  for any connected manifold $B$ has dimension:
  \begin{equation}
    \sum_{[A] \in \mathcal{A}/G} |\{\text{irreps of} \opname{Aut}(A)\}|
  \end{equation}
  The sum is over equivalence classes of connections on $B$.
\end{proposition}
\begin{proof}
  This is just a special case of the general fact about representation
  categories for groupoids.

  The groupoid $\fc{B}$ is equivalent to its skeleton $S$, whose
  objects are the gauge equivalence classes of connections on $B$. The
  morphisms are the stabilizer groups at each object. That is, they
  are exactly the gauge transformations which fix a representative of
  each class. Then $[S,\V] \simeq \Z{B}$, but $[S,\V]$ is a KV vector
  space, hence equivalent to some $\V^m$, where $m$ is the number of
  non-isomorphic simple objects.

  A functor $F : S \ra \V$ assigns a vector space to each object $[A]$
  (equivalence class of connections), carrying a representation of
  $\opname{Aut}(A)$.  Two functors giving inequivalent representations
  cannot have any nontrivial natural transformation between them, by
  Schur's lemma. On the other hand, any representation of
  $\opname{Aut}(A)$ is a direct sum of irreducible representations. So
  a choice of simple object in $\Z{B}$ amounts to a choice of $[A]$,
  and an irreducible representation of $\opname{Aut}(A)$. The
  statement follows immediately.
\end{proof}

Our construction works, in principle, for manifolds of any dimension,
though computations become more involved in higher dimension as one
might expect.  Next we consider some particular examples.

\begin{example}\label{ex:ZonS1}
  All 1-dimensional manifolds are equivalent to a collection of
  circles, so to understand the effect on 1-manifolds, it suffices to
  know that the 2-vector space assigned to the circle $S^1$ by $Z_G$
  is:
  \begin{equation}
    \Z{S^1} = \FV \circ \fc{S^1} = Rep(\fc{S^1})
  \end{equation}

  The fundamental group of the circle is $\mathbb{Z}$, and
  $\Pi_1(S^1)$ is thus equivalent to $\mathbb{Z}$ as a one-object
  category. A functor $g: \mathbb{Z} \ra G$ is determined by $g(1) \in
  G$, which we also denote $g$ for convenience. A natural
  transformation is a conjugacy relation: $h: g \rightarrow g'$
  assigns to the single object in $\mathbb{Z}$ a morphism $h \in G$,
  and the naturality condition that $g'h = hg$, or simply $g' =
  hgh^{-1}$.

  Thus, $\fc{S^1}$ is equivalent to the groupoid whose objects are $g
  \in G$, and whose morphisms are conjugacy relations between
  elements.  This is the transformation groupoid of the adjoint action
  of the group $G$ on itself, also known as $G$ ``weakly modulo'' $G$,
  or $G \wquot G$.  This is also the discrete form of the \textit{loop
    groupoid} $LG$, as summarized in the account by Willerton
  \cite{willerton}).

  Finally, the 2-vector space corresponding to the circle is the
  category of representations of $G \wquot G$.
  \begin{equation}
    Z_G(S^1) = Rep(G \wquot G)
  \end{equation}
  is generated by a basis of irreducible objects, the elements of
  which are labelled by pairs: a conjugacy class $[g]$ in $G$, and an
  irreducible representation of $Aut([g])$. All representations are
  direct sums of these.  One can also think of an object of $Z_g(S^1)$
  as a $G$-equivariant vector bundle on the groupoid of connections on
  $S^1$, which is just $G$ .

  The skeleton of $G \wquot G$ has as objects the conjugacy classes of
  $G$, and each object has $Aut(g) < G$, the stabilizer subgroup of
  $g$, which is the centralizer $C_g$.

  Since a general 1-manifold is isomorphic to $S^1 \cup \dots \cup
  S^1$, the effect of $Z_G$ then just amounts to:
  \begin{align}
     Z_G ( S^1 \cup \dots \cup S^1 ) 
     & = Z_G(S^1) \otimes \dots \otimes Z_G(S^1) \\
     & = Rep(G \wquot G) \otimes \dots \otimes Rep(G \wquot G)
  \end{align}
  (The monoidal operation $\otimes$ is the Deligne tensor product for
  Abelian categories.  It is analogous to the tensor product for
  modules or vector spaces and is dual to the internal $Hom$ functor,
  so that $A \otimes B$ is a representing object for bi-2-linear
  functors out of $A \times B$.)
\end{example}

Let us briefly consider the case where objects are 2-dimensional
manifolds (as in the 4D TQFT).  We will not study the 4D ETQFT in
detail, but this will illustrate that the same construction will work
in that case.

\begin{example}
  Consider the torus $T^2 = S^1 \times S^1$. We want to find
  \begin{equation}
    Z_G(T^2) = \Z{T^2}
  \end{equation}
  The category $\fc{T^2}$ is again the category of functor $\Pi_1(T^2)
  \ra G$ and natural transformations. The groupoid $\Pi_1(T^2)$ is
  equivalent to its skeleton, namely the fundamental group of
  $T^2$. This is just $\mathbb{Z}^2$.  So the functor $F \in [
  \mathbb{Z}^2 , G ]$ is a group homomorphism from $\mathbb{Z}^2$ to
  $G$, and thus determined by the images of the two generators $(1,0)$
  and $(0,1)$.  Since $\mathbb{Z}^2$ is abelian, the images $g_1 =
  F(1,0)$ and $g_2 = F(0,1)$ must commute.

  So the objects of $\fc{T^2}$ are indexed by commuting pairs of
  elements $(g_1,g_2) \in G^2$. (We note here that in the case of a
  topological group, this is a space of some interest in itself; see
  e.g. \cite{ademcohen}. In the discrete case, this is simply a set.)

  A natural transformation $g : F \ra F'$ assigns to the single object
  $\star$ of $\mathbb{Z}^2$ a morphism in $G$---that is, a group
  element $h$. This must satisfy the naturality condition $h F(a)
  h^{-1} = F'(a)$ for all $a$. This will be true for all $a$ in
  $\mathbb{Z}^2$ as long as it is true for $(1,0)$ and $(0,1)$.

  In other words, for functors $F$ and $F'$ represented by $(g_1,g_2)
  \in G^2$ and $(g'_1,g'_2) \in G^2$, the natural transformations
  $\alpha : F \ra F'$ are represented by group elements $h \in G$
  which act in both components at once, so $(h^{-1}g_1h,h^{-1}g_2h) =
  (g'_1,g'_2)$.

  So we have that the groupoid $\fc{T^2}$ is equivalent to $A \wquot
  G$, where $A = \{(g_1,g_2) \in G^2 : g_1g_2 = g_2g_1\}$, and the
  action of $G$ on $A$ comes from the action on $G^2$ as above.

  So the 2-vector space $Z_G(T^2)$ is just the category of
  $\V$-presheaves on $A$, equivariant under the given action of $G$.
  This assigns a vector space to each connection $(g_1,g_2)$ on $T^2$,
  and an isomorphism of these vector spaces for each gauge
  transformation $h : (g_1,g_2) \mapsto (h^{-1}g_1h,h^{-1}g_2h)$.
  Equivalently (taking a skeleton of this), we could say it gives a
  vector space for each equivalence class $[(g_1,g_2)] \in G^2$ under
  simultaneous conjugation, and a representation of $G$ on this vector
  space.

  A similar pattern will apply to a 2-dimensional surface of genus
  $k$.
\end{example}

\subsection{$Z_G$ on Cobordisms}

To help clarify the construction we have described, we consider some
examples for particular cobordisms, and particular groups $G$.

We should emphasize here that although we are primarily considering
the case $n=3$ in order to establish a link with the DW
theory, we will express our results generically.

In particular, Turaev established \cite{turaev-MTC} that a
(2+1)-dimensional TQFT is determined by a modular tensor category
(MTC) $\catname{C}$ (i.e. monoidal category with a modular
structure). Such a category has a finite set of generators (the number
of generators is the ``rank'' of the MTC), and the monoidal structure
is determined by the matrices representing $Z_G(Y)$.

In the framework we have been describing, $\catname{C} = Z_G(S^1)$ is
the 2-vector space for the circle. The monoidal structure for
$\catname{C}$ as a MTC turns out to be given by the value of $Z_G$ on
the ``pair of pants'' cobordism, or ``trinion'', $Y : S^1 \sqcup S^1
\ra S^1$. This depicts two circles coalescing into one circle, by a
cobordism which is topologically a sphere with the interiors of three
disks removed. It therefore determines a monoidal structure by:
\begin{equation}
  Z_G(Y) : \catname{C} \otimes \catname{C} \ra \catname{C}
\end{equation}
Now, up to diffeomorphism, all 2-dimensional cobordisms can be written
as composites of some number of copies of $Y$ and its dual, so knowing
this monoidal structure completely determines the effect of the TQFT
on morphisms. It is common to describe this using the monoidal
structure directly.

However, our general framework can be applied to other values of $n$,
for which there is no such simple ``normal form'' for
cobordisms. However, we can still compute 2-linear maps in these
cases. It is common to describe a MTC $\catname{C}$ as a monoidal
category, giving fusion rules for the ``tensor product'' functor
$Z_G{Y}$. However, we prefer here simply to use a standard matrix
representation of a 2-linear map, since these techniques would apply
equally in higher dimensions where no such presentation exists, though
of course when $n=3$ the two approaches are equivalent.

\subsubsection{Example: The Pair of Pants}

Applying the monoidal functor $Z_G$ to the pair of pants gives a 2-linear map:
\begin{equation}
  Z_G(Y) : Z_G(S^1) \otimes Z_G(S^1) \rightarrow Z_G(S^1)
\end{equation}

The basis of generating objects for the two 2-vector spaces are the
irreducible representations of the corresponding groupoid. Given
irreducibles $V$, and $W$, the coefficients of the 2-linear map
$Z_G(Y)$ are, by Frobenius reciprocity (see \cite{gpd2vect}):
\begin{equation}
  Z_G(Y)_{V,W} \cong \opname{hom}_{Rep(\fc{Y})}(p_1^{\ast}(V),p_2^{\ast}(W))
\end{equation}
That is, one pulls back the basis 2-vectors to give representations
$p_1^{\ast}(V)$ and $p_2^{\ast}(W)$ of $\fc{Y}$, the middle groupoid
of the span. The coefficient is the ``inner product''---the internal
hom, which is the space of intertwiners between the two pulled-back
representations. By Frobenius reciprocity, this amounts to counting
the multiplicities of each generating irreducible in the target 2-vector space in
the image of the chosen generating irreducible in the source.

Since the fundamental groupoid of $S^1 \cup S^1$ is just $\Pi_1(S^1)
\cup \Pi_1(S^1)$, (a disjoint union of two copies of $\Pi_1(S^1) \cong
\mathbb{Z}$), a functor into $G$ then amounts to two choices $g, g'
\in G$. A gauge transformation amounts to a conjugation by some $h \in
G$ at each of two chosen base points, one in each
component:
\begin{eqnarray}
  \fc{S^1 \cup S^1}  & \cong &  (G \times G) \wquot (G \times G) \\
  \nonumber  &\cong &  (G \wquot G)^2
\end{eqnarray}
where $G \times G$ acts on itself by conjugation component-wise. This
is illustrated in Figure \ref{fig:pantsconnection}.  The connection on
$Y$ has holonomies $g$ and $g'$ around the two holes. On $S^1 \cup
S^1$, this restricts to a connection with holonomies $g$ and $g'$
respectively, and on $S^1$ to the product (since the outside $S^1$ is
homotopic to the composite of the two loops).

\begin{figure}[h]
  \begin{center}
    \includegraphics{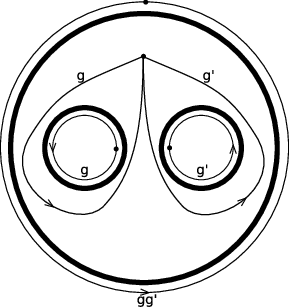}
  \end{center}
  \caption{Connection for Pants\label{fig:pantsconnection}}
\end{figure}

The fundamental groupoid is equivalent to $\pi_1(Y) =
F(\gamma_1,\gamma_2)$, the free group on two generators. A functor in
$[\Pi_1(Y),G]$ thus amounts to a pair of elements $(g,g')$, the images
of the two generators. A gauge transformation amounts to conjugation
at the single object (a chosen base point in $Y$). So we have the
groupoid:
\begin{equation}
  \fc{S} \cong (G \times G) \wquot G
\end{equation}
in which $G$ acts on $G \times G$ by conjugation in both components at
once. Thus we have the span:
\begin{equation}\label{eq:pantspan}
  \xymatrix{
    & (G \times G)\wquot G \ar[dl]_{p_1} \ar[dr]^{p_2} & \\
    (G \wquot G)^2 & & G \wquot G
  }
\end{equation}
Both projections are restrictions of a connection on $Y$ to the
corresponding connection on the components of the boundary. It is
clear from the figure that on objects:
\begin{align}
p_1 : & (g_1, g_2) \Rightarrow (g_1, g_2) \\
p_2 : & (g,g') \Rightarrow gg'
\end{align}
and on morphisms
\begin{align}
  p_1 : & (h : (g_1,g_2) \ra (hg_1g^{-1},hg_2h^{-1}) ) \Rightarrow ( (h,h) : (g_1,g_2) \ra (hg_1h^{-1},hg_2h^{-1}) )\\
  p_2 : & (h : (g,g') \ra (hgh^{-1},hg'h^{-1}) ) \Rightarrow ( h : gg'
  \ra hgg'h^{-1})
\end{align}

The classes of connections on $Y$ are of the form $[(g,g')]$ for $g,
g' \in G$, and the class is an equivalence class modulo gauge
transformations, conjugating by $(h,h)$.  The classes for $S^1 \cup
S^1$ are of the form $([g],[g'])$, since equivalence is by conjugation
by $(h,h') \in G^2$. So connections which are gauge equivalent on $S^1
\cup S^1$ may be restrictions of inequivalent connections on $Y$.

Finally, suppose we have a functor $f : \fc{S^1 \cup S^1} \ra \V$:
that is, a representation of $\fc{S^1} \times \fc{S^1}$, and transport
it by $Z_G(Y) = (p_2)_{\ast} \circ (p_1)^{\ast}$. That is, we first
pull back along $p_1$ from $S^1 \cup S^1$ to $Y$, then push forward
along $p_2$ to $S^1$.

Now, an irreducible representation of $\fc{S^1} \times \fc{S^1}$
amounts to a pair of irreducible representations of $\fc{S^1}$. Each
one amounts to a choice of isomorphism class $[g]$ in $G \wquot G$,
and a representation $V$ of $Aut(g)$.  Call these $([g],V)$ and
$([g'],V')$. Then the pair $([g],[g'])$ is the image of any $[(g,g')]$
in $(G \times G) \wquot G$ for some pair $(g,g')$ representing
$([g],[g'])$. This will then be pushed down by $m$ to a representation
of $Aut(g_1)$, where $g_1 = gg'$. There may be more than one
$[(g,g')]$ for which this holds for a given $[g_1]$.  In fact,
following \cite{gpd2vect}, we can then find that the image is:
\begin{equation}\label{ZGimage}
  (\pi_2)_{\ast} \circ \pi_1^{\ast}(g_1) \cong \bigoplus_{(g,g') \in
    \{ ([g],[g']) | gg' = g_1 \} } \mathbb{C}[\opname{Aut}(g_1)]
  \otimes_{\mathbb{C}[\opname{Aut}(g,g')]} (V \otimes V')
\end{equation}
where the direct sum is over all choices of $(g,g')$ representing
$([g],[g'])$ and satisfying $gg'=g_1$, up to equivalence in $(G \times
G)\wquot G$. The action of $G$ on each component is as we have
described. On morphisms, we get the direct sum of the isomorphisms
between these copies of $\mathbb{C}$. (This formula is a special case
of the general form for Kan extension in an enriched category - see
e.g. \cite{dubuc}).

Now suppose $G$ is abelian. Then each element is in a unique conjugacy
class. The irreducible representations of $G$ are 1-dimensional,
forming $\hat{G}$, the character group of $G$. For abelian groups,
this is just $\hat{G} \cong G$. So the simple representations are
labelled by $G \times G$. This and general structure theorems for
abelian groups make this case simple.

\begin{lemma}
  If $G = G_1 \oplus \dots \oplus G_n$ is the direct sum of $n$ abelian
  groups, then $\fc{S^1}$ is:
  \begin{equation}
    G \wquot G \cong (G_1 \wquot G_1) \times \dots \times (G_n \wquot G_n)
  \end{equation}
\end{lemma}

This follows directly from the definitions, and, combined with
structure theorems for abelian groups and simple calculations for
cyclic groups, gives the simple result:

\begin{proposition}
  For $G$ a finite abelian group, the 2-linear map $Z_G(Y)$ is given
  by group operation $+$ in $G \times G$:
  \begin{equation}
    Z_G(Y) : ((g_1,g_2),(g'_1,g'_2)) \mapsto (g_1+g'_1,g_2+g'_2)
  \end{equation}
\end{proposition}

The case of $G$ nonabelian is more interesting, since there may be
more than one representative of $([g],[g'])$ contributing to a given
term, and the stabilizer groups are different for different
objects. Since the matrix form of 2-linear maps $Z_G(Y)$ are in
general rather large even for fairly small $G$, so, we will here only
illustrate certain interesting components in some simple cases, namely
$G = S_3$ and $G=S_4$.  Since in general these matrices are quite
large, we will simply find a few blocks.

\begin{example}\label{ex:zgs3}
  First, find $Z_{S_3}(S^1) = Rep(S_3 \wquot S_3)$. As usual, this is
  generated by irreducible objects labelled by $([g],\rho)$, where
  $[g]$ is a conjugacy class in $G = S_3$, and $\rho$ is an
  irreducible representation of $Stab(g) \subset G = S_3$. The
  groupoid $S_3 \wquot S_3$ has stabilizer groups and irreducible
  representation as in Table \ref{table:s3class}.
  \begin{table}\label{table:s3class}[h]
    \begin{tabular}{|l|l|l|}
      \hline
      Class $[g]$ & Group $Stab(g)$ & Representations \\
      \hline
      1 (identity) & $S_3$ & $\mathbb{C}$, $\Gamma =$ {\tiny{$\yng(2,1)$}}, $\sigma$ \\
      \hline
      [t] (transposition) & $\mathbb{Z}_2 = \{ 1, t \}$ & $\mathbb{C}$ and $\sigma$ \\
      \hline
      [s] (3-cycle) & $\mathbb{Z}_3 = \{ 1 , s , s^2 \}$ & $\mathbb{C}$, $\phi$, $\phi^2$ \\
      \hline
    \end{tabular}
    \caption{Generators of $Z_{S_3}(S^1)$}
  \end{table}
  For $[t]$ and $[s]$, the stabilizer groups are abelian, so for
  $\mathbb{Z}_2$ we have the trivial and sign representations, and for
  $[s]$ we have the irreducible representations of $\mathbb{Z}_3$ on
  $\mathbb{C}$ as before.  For $S_3$, the three irreducible reps are
  labelled by three-block Young tableaux (see e.g. \cite{sternberg}),
  though these include the trivial representation $\mathbb{C} =$
  {\tiny{$\yng(3)$}} and the sign representation $\sigma =$
  {\tiny{$\yng(1,1,1)$}}. (The remaining irreducible representation,
  $\Gamma = ${\tiny{$\yng(2,1)$}} is the 2-dimensional representation
  of $S_3$ given by the action of $S_3$ on the vertices of a
  triangle.)

  Then $Z_{S_3}(Y) : Z_{S_3}(S^1 \coprod S^1) \rightarrow
  Z_{S_3}(S^1)$ is a 2-linear map taking representations of $(S_3
  \wquot S_3)^2$ to representations of $S_3 \wquot S_3$, which are as
  described above. An irreducible representation of $(S_3 \wquot
  S_3)^2$ is labelled by a pair
  $\bigl{(}([g],\rho),([g'],\rho')\bigr{)}$ of irreducible
  representations of $S_3 \wquot S_3$.

  The functor $Z_{S_3}(Y)$ can then be described by a (64-by-8) matrix
  of vector spaces. The entries are given by Frobenius reciprocity,
  pulling representations back along $m$ and $\Delta$ to
  representations of $\fc{Y} = (S_3 \times S_3) \wquot S_3$. So in
  particular, we get a sum over isomorphism classes in $\fc{Y}$. These
  are given in Table \ref{table:s3class}.

  \begin{table}\label{table:s3s3class}[h]
    \begin{tabular}{|c|c|c|c|}
      \hline
      $[(g_1,g_2)]$ & $Stab([(g_1,g_2)])$ & $\opname{Im}(m)$ & $\opname{Im}(\Delta)$ \\
      \hline
      $[(1,1)]$   & $S_3$ & $[1]$ & $([1],[1])$ \\
      \hline
      $[(1,t)]$   & $\mathbb{Z}_2$ & $[t]$ & $([1],[t])$ \\
      \hline
      $[(1,s)]$   & $\mathbb{Z}_3$ & $[s]$ & $([1],[s])$ \\
      \hline
      $[(t,1)]$   & $\mathbb{Z}_2$ & $[t]$ & $([t],[1])$ \\
      \hline
      $[(t,t)]$   & $\mathbb{Z}_2$ & $[1]$ & $([t],[t])$ \\
      \hline
      $[(t,t')]$  & $1$ & $[s]$ & $([t],[t])$ \\
      \hline
      $[(t,s)]$   & $1$ & $[t]$ & $([t],[s])$ \\
      \hline
      $[(s,1)]$   & $\mathbb{Z}_3$ & $[s]$ & $([s],[1])$ \\
      \hline
      $[(s,t)]$   & $1$ & $[t]$ & $([s],[t])$ \\
      \hline
      $[(s,s)]$   & $\mathbb{Z}_3$ & $[s]$ & $([s],[s])$ \\
      \hline
      $[(s,s^2)]$ & $\mathbb{Z}_3$ & $[1]$ & $([s],[s])$ \\
      \hline
    \end{tabular}
    \caption{Isomorphism Classes of of $(S_3 \times S_3) \wquot S_3$}
  \end{table}
  Since both elements of a pair $(g_1,g_2)$ are conjugated by the same
  $g$ in this quotient action, we can distinguish cycles and
  permutations, as in $([s],[s^2])$. Thus, there are two possible
  preimages of $([t],[t])$, depending on whether the two permutations
  $t$ and $t'$ are the same, and similarly for $([s],[s])$.

  So then we have that in matrix form
  $Z_{S_3}(Y)_{([g],\rho),(([g_1],\rho_1),([g_2],\rho_2))}$ is given
  by a direct sum over the isomorphism classes in $(S_3 \times S_3)
  \wquot S_3$ lying over $[g]$ by $m$ and over $([g_1],[g_2])$ by
  $\Delta$ (that is, non-conjugate pairs with $g_2 g_1 = g$). The
  coefficients for particular representations show the multiplicity of
  $\rho$ in the image of $(\rho_1,\rho_2)$. Unlike the abelian case,
  there are nontrivial coefficients from $([t],[t])$ to two different
  elements, $[s]$ and $[1]$.

  For example, we now find the block of $Z_{S_3}(Y)$ corresponding to
  the objects $([t],[t])$ and $[1]$.  There is a single object in
  $(S_3 \times S_3) \wquot S_3$ lying over these objects, namely
  $[(t,t)]$.  Restricting to these objects, we have the span of
  automorphism groups:
  \begin{equation}
    \mathbb{Z}_2 \times \mathbb{Z}_2 \stackrel{\Delta}{\la} \mathbb{Z}_2 \stackrel{i}{\ra} S_3
  \end{equation}
  (The map $i : \mathbb{Z}_2 \ra S_3$ is the injection homomorphism
  which takes the non-identity element of $\mathbb{Z}_2$ to $t \in
  S_3$.)  We can calculate the block of $Z_{S_3}(Y)$ with indices
  given by the irreducible representations of these groups, namely $\{
  (\mathbb{C},\mathbb{C}), (\mathbb{C},\sigma), (\sigma,\mathbb{C}),
  (\sigma,\sigma) \}$, and $\{ \mathbb{C},\Gamma,\sigma \}$,
  respectively.

  We find these by pulling back each representation to $\mathbb{Z}_2$
  along $\Delta$ or $i$, and using Schur's lemma.  One can easily find
  the block to be (using integers to represent vector spaces):
  \begin{equation}
    Z_{\mathbb{S}_3}(Y)_{1,([t],[t])} =
    \begin{pmatrix}
      0 & 1 & 1 \\
      1 & 1 & 0 \\
      1 & 1 & 0 \\
      0 & 1 & 1
    \end{pmatrix}
  \end{equation}
\end{example}

In the example above, we see that the matrix form of $Z_G(Y)$ need not
be a multiplication matrix for a group, as it is for the abelian case,
essentially because the image of an irreducible may not itself be
irreducible.  However this example is still special in that only one
object in the middle groupoid $G \times G \wquot G$ contributes to any
given matrix entry.  This is not true in general.  The following
(abbreviated) example illustrates the point in a more general
framework than the above.

\begin{example}
  If $G = S_4$, then $\fc{S^1} = S_4 \wquot S_4$ has isomorphism
  classes given by the conjugacy classes of permutations of 4
  elements.  These are classified by 4-box Young diagrams, of which
  there are five:{\tiny{
  \begin{equation}
    \yng(4), \yng(3,1), \yng(2,2), \yng(2,1,1), \yng(1,1,1,1)
  \end{equation}}}
  In the same way, $\fc{S^1 \cup S^1}$ has isomorphism classes given
  by pairs of such diagrams.

  As in Example \ref{ex:zgs3}, we find a single block of $Z_{S_4}(Y)$,
  namely the block corresponding to $( (${\tiny{$\yng(3,1)$}}$ ,
  ${\tiny{$\yng(3,1)$}}$) , ${\tiny{$\yng(3,1)$}}$ ) $.  Here we are
  using the diagram which corresponds to the conjugacy class of a
  3-cycle (i.e. a permutation of four elements with one fixed point).
  In the usual cycle notation for permutations, this object is
  $[(123)(4)]$.

  The point now is that there are two classes in $(S_4 \times S_4)
  \wquot S_4$ which project to $[((123)(4),(123)(4))]$ under $\Delta$
  and $[(123)(4)]$ under the multiplication map $m$.  That is, there
  are two conjugacy classes of pairs of 3-cycles whose product is a
  3-cycle.  Representatives of these two classes are: $a = ( (123)(4),
  (123)(4) )$, where the product is $m(a)=(132)(4)$; and $b=( (123)(4) ,
  (243)(1) )$.  where the product is $m(b)=(143)(2)$.  It is
  straightforward to check these are the only cases.

  Now, $Aut((123)(4)) \cong \mathbb{Z}_3$: precisely the powers of
  this 3-cycle stabilize it under conjugation ($4$ must be a fixed
  point of any $\pi \in Aut( (123)(4) )$, and a transposition would
  change the cycle).  So, since $a$ consists of two copies of this
  cycle, $Aut(a) \cong \mathbb{Z}_3$ also.  On the other hand, $Aut(b)
  = \{ Id \}$: no permutation stabilizes both permutations in the pair
  $b$.

  Thus, to find the $( (${\tiny{$\yng(3,1)$}}$ ,
  ${\tiny{$\yng(3,1)$}}$) , ${\tiny{$\yng(3,1)$}}$ ) $ block of
  $Z_{S_4}(Y)$, we take a direct sum over $a$ and $b$ of the
  restriction-induction functors.  These come from two spans of
  automorphism groups from $\mathbb{Z}_3^2$ to $\mathbb{Z}_3$.  One
  can check that at $a$ we have the span:
  \begin{equation}
    \mathbb{Z}_3 ^2 \stackrel{\Delta}{\la} \mathbb{Z}_3 \stackrel{id}{\ra} \mathbb{Z}_3
  \end{equation}
  The corresponding 2-linear map is just the multiplication map.  On
  the other hand, at $b$ we have the span:
  \begin{equation}
    \mathbb{Z}_3 ^2 \stackrel{i}{\la} \{ Id \} \stackrel{i}{\ra} \mathbb{Z}_3
  \end{equation}
  In each case, the maps are the inclusions of the identity. Since the
  representations of $\mathbb{Z}_3$ (i.e. characters in
  $\hat{\mathbb{Z}_3}$) all pull back to the unique, trivial,
  representation of $\{ Id \}$, Schur's lemma says the resulting
  matrix has $\mathbb{C}$ in each component.  So taking the direct sum
  over $a$ and $b$, we find the block is:
  \begin{equation}
    \begin{pmatrix}
      2 & 1 & 1 & 1 & 2 & 1 & 1 & 1 & 2 \\
      1 & 2 & 1 & 1 & 1 & 2 & 2 & 1 & 1 \\
      1 & 1 & 2 & 2 & 1 & 1 & 1 & 2 & 1
    \end{pmatrix}
  \end{equation}
  The other blocks are all found in a similar way (though we note that
  this is the only block for the case $G = S_4$ where more than one
  object appears in the direct sum.)
\end{example}

The final aspect of our weak 2-functor is its effect on
2-morphisms.

\subsection{$Z_G$ on Cobordisms of Cobordisms}\label{sec:ZGonCobCob}

Now we consider the situation of a cobordism $M$ between cobordisms.

Recall that we are taking advantage of the construction
\cite{dblbicat} of a double bicategory of cobordisms $\nCob$,
consistent with the cubical approach of \cite{grandis1,
  grandis3}. This is a (weak) cubical 2-category, having horizontal
and vertical 1-morphisms which are both cobordisms, and horizontal and
vertical 2-morphisms which are diffeomorphisms between these. A square
in $\nCob$ is a cobordism between cobordisms. We may intuitively
understand the horizontal 1-morphisms at the top and bottom as the
(vertical) source and target. The (horizontal) boundaries change, by
the cobordisms which are the vertical 1-morphisms at the sides of the
square. The situation is, of course, completely symmetric, so we may
exchange the roles of horizontal and vertical in this interpretation.

The cubical picture is only a slightly more general picture than the
more common bicategory view of cobordisms. The usual bicategory of
cobordisms has as 2-morphisms which correspond exactly to those
squares with horizontal boundaries constrained to be identity
cobordisms (that is, the boundaries do not change). However, as
discussed in \cite{dblbicat}, it is possible to reduce $\nCob$ to a
bicategory, which is equivalent to the usual bicategory of
cobordisms. This uses the fact that horizontal and vertical morphisms
are both cobordisms, they can be composed with each other, and squares
can be taken to 2-morphisms in the bicategory of cobordisms. The only
information lost in this process is precisely how a given cobordism
factors into horizontal and vertical parts.

In the construction we give here, in order to make the connection
between the extended TQFT we construct and the DW model,
we will use the bicategory $\iiV$ as our target. Just because this is
a bicategory, and because we want to make use of the quantization
2-functor $\Lambda$ from \cite{gpd2vect} we use this reduction of
$\nCob$ to a bicategory.

In figure \ref{fig:2morphism} we show an example, which can be
construed as taking a pair of pants $Y$ to its reversed version
$Y^{\dagger}$.

\begin{figure}[h]
  \begin{center}
    \includegraphics{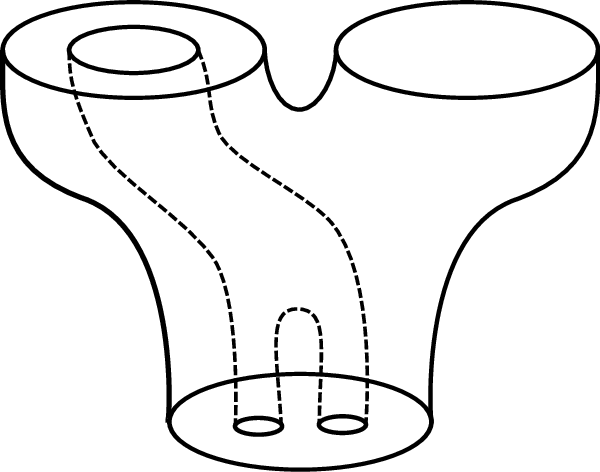}
  \end{center}
  \caption{\label{fig:2morphism}A 2-Morphism in $\iiiCob$}
\end{figure}
This clearly depicts a square in the double bicategory, since both
source and target change in this process. 

Our construction will then use the fact that there is an inclusion
\begin{equation}
  \nCob \ra \Cosp^2(\ManCorn)
\end{equation}
Here, a cobordism becomes a span of
inclusion maps. Thus the (collarable) inclusions of these boundary
components, and the corners, into this cobordism give the following
square in $\Cosp^2(\ManCorn)$:
\begin{equation}\label{eq:2morphismcosp}
  \xymatrix{
    S^1 \ar[r]^{i_A} \ar[d]_{i_1} & (A \coprod D) \ar[d]_{\iota_1} & S^1 \coprod S^1 \ar[l]_{i'_A \otimes i_D} \ar[d]^{i_2}\\
    Y \ar[r]^{\iota_3} & M & \ar[l]_{\iota_4} Y \\
    S^1 \coprod S^1 \ar[r]_{i_2} \ar[u]^{i_2} & Y \ar[u]^{\iota_2} & S^1 \ar[u]_{i_1} \ar[l]^{i_1}
  }
\end{equation}
Here, $A$ is the annulus and $D$ the disk in the top horizontal
cobordism, the $Y$ are instances of the pair of pants, and $M$ is the
whole 3-dimensional manifold with corners. The leftmost vertical
cospan is the inner cobordism, and the rightmost is the outer. The
maps come from the obvious inclusions.

To get a corresponding 2-morphism in the bicategory $\iiiCob$, one
must make some arbitrary choices, to choose the source and target
objects from the four corners of the square. A corollary of this
choice is that we lose the extra information encoded in the other two
corners of the square, and the distinction between horizontal and
vertical (hence change of boundary along the cobordism $M$).

The convention we adopt here is to take the source as the upper left
and the target as the lower right. The source and target morphisms are
then found by composing around the corners of the square. This
convention turns $M$ into a cobordism between two cobordisms, each of
which goes from $S^1$ to $S^1$. That is, $M$ becomes a cospan of
cospan maps, forming a 2-morphism between the two cobordisms:
\begin{equation}\label{eq:2morphismspancomposed}
  Y \circ (A \coprod D) \ra M \la Y \circ Y^{\dagger}
\end{equation}
as illustrated in Figure \ref{fig:2morbicat}. Here, $Y^{\dagger}$ is
the adjoint of $Y$ as a cobordism, namely $Y$ with direction reversed.
This step of the construction of $Z_G$ is slightly awkward since
$\nCob$ is most generally a double bicategory, that is, a weak cubical
2-category, and $\iiV$ is most naturally a bicategory (that is, a weak
globular 2-category).  As mentioned in Section \ref{sec:ncob}, this
bicategory is equivalent to the more commonly used form of $\nCob$.
We have chosen this method of reconciling them rather than the
alternate approach of using some cubical version of $\iiV$ for several
reasons: cubical $n$-categories are simply less familiar, the
(globular) bicategory $\iiV$ is already in common use, and the
quantization functor $\Lambda$ is adapted to it and has a
straightforward relation to the groupoidification program.

\begin{figure}[h]
  \begin{center}
    \includegraphics{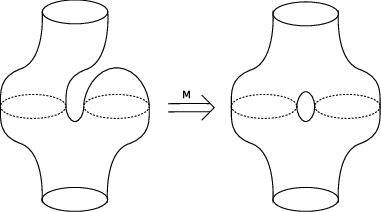}
  \end{center}
  \caption{\label{fig:2morbicat}The Same 2-Morphism, In A Bicategory}
\end{figure}

The left-hand composite $S$ just amounts to the identity cobordism for
$S^1$. So in particular, we can also regard this calculation as giving
a cobordism from the identity cobordism on $S^1$ (the cylinder) to $Y
\circ Y^{\dagger}$. The latter topologically a genus-1 surface with
removed disks at the input and output boundaries. Applying $\fc{-}$
gives a span of connection groupoids, all $G \wquot G$, with two
identity maps. The right-hand cobordism is the composite $S' =
Y^{\dagger} \circ Y$.

Now, the point of using this inclusion of $\nCob$ into
$\Cosp^2(\ManCorn)$ is that we can then extend the contravariant
functor $\fc{-} : \ManCorn \ra \Gpd$ to give a functor $\nCob \ra
\Span(\Gpd)$. Each object in a cospan then gives a groupoid, and the
contravariant turns a cospan into a span. Likewise, for 2-morphisms,
$\fc{-}$ gives a span of spans of groupoids of connections, as in
(\ref{eq:FV2mor}). We can then combine this 

In the case where objects in $\nCob$ are empty manifolds $\emptyset$,
cobordisms between two empty sets are themselves manifolds (with empty
boundary), and cobordisms between these have boundary, but no
nontrivial corners. So we have just a cobordism from one manifold to
another. It is reasonable to expect that in this case, the extended
TQFT based on a group $G$ should give results equivalent to those
obtained from a TQFT based on the same group, suitably
reinterpreted. This is indeed the case.

\begin{example}\label{ex:closedTQFT}Suppose we have two
  cobordisms $S$ and $S'$ from $\emptyset$ to $\emptyset$, and a
  cobordism with (empty!) corners $M : S \ra S'$. In fact, $M$ should
  be thought of as a cobordism between manifolds, in a precisely
  analogous way that $Z_G(S)$ can be thought of as a TQFT giving a
  vector space for the manifold $S$.

  In particular, we have
  \begin{equation}
    Z_G(S) \cong (- \otimes \mathbb{C}^k)
  \end{equation}
  and
  \begin{equation}
    Z_G(S') \cong (- \otimes \mathbb{C}^{k'})
  \end{equation}
  The $k$ and $k'$ are the number of isomorphism classes of
  connections on $S$ and $S'$ respectively. If we think of these as
  being vector spaces $\mathbb{C}^k$ and $\mathbb{C}^{k'}$ assigned by
  a TQFT, then a cobordism should assign a linear map $T :
  \mathbb{C}^k \ra \mathbb{C}^{k'}$. Indeed, such a linear map will
  give rise to a natural transformation from $Z_G(S)$ to $Z_G(S')$ by
  giving, for any objects $V \in \V$ on the left side of the diagram,
  the map $1_V \otimes T$ on the right side. Moreover, all such
  natural transformations arise this way.

  A cobordism between cobordisms gives rise to a natural
  transformation:
  \begin{equation}
    \xymatrix{
      \V \ar@/^2pc/[rr]^{(\pi_2)_{\ast} \circ \pi_1^{\ast}}="0" \ar@/_2pc/[rr]_{(\pi'_2)_{\ast} \circ (\pi'_1)^{\ast}}="1" & & \V \\
      {\ar@{=>}"0"+<0ex,-2.5ex> ;"1"+<0ex,+2.5ex>^{Z(M)}}
    }
  \end{equation}

  As discussed in \cite{gpd2vect}, this reduces to the
  groupoidification functor
  \begin{equation}
    D : \Span(\Gpd) \ra \V
  \end{equation}
  and our construction just yields a TQFT. That is, each 2-linear map
  can then be described as a 1-by-1 matrix of vector spaces (that is,
  a vector space), and the natural transformations are just described
  in this one component by a single linear map.
\end{example}

Now we look, in the 3D case, at a more general 2-morphism $M$ in
$\nCob$ and find $Z_G(M)$ for a general $G$.

\begin{example}
  Next, consider the cobordism $M$ between cobordisms from Figure
  \ref{fig:2morphism}, discussed above. The corresponding square
  (\ref{eq:2morphismcosp}).
 
  To compute $Z_G(M)$, we should convert $M$ to a 2-morphism in the
  bicategory $\iiiCob$, then apply $Z_G$. However, consider first the
  effect of $\fc{-}$ on (\ref{eq:2morphismcosp}), giving (up to
  equivalence) the following double span of groupoids:
  \begin{equation}\label{eq:2morphismspan}
    \xymatrix{
      G \wquot G & (G \wquot G) \otimes 1 \ar[l]_{id} \ar[r]^{id \otimes !} & (G \wquot G) ^2 \\
      G^2 \wquot G \ar[u]^{m} \ar[d]_{\Delta} & G^2 \wquot G \ar[l]_{id} \ar[r]^{(m,1)} \ar[u]^{(m,!)} \ar[d]_{id} & G^2 \wquot G \ar[u]_{\Delta} \ar[d]^{m} \\
      (G \wquot G)^2 & G^2 \wquot G \ar[l]^{\Delta} \ar[r]_{m} & G \wquot G
    }
  \end{equation}

  The corresponding 2-morphism in $\Span(\Gpd)$, can be found either
  by applying $\fc{-}$ directly to the cospan of cospans obtained from
  (\ref{eq:2morphismcosp}) by composing around corners, or else by
  finding (\ref{eq:2morphismspan}), and composing here. These are
  equivalent since $\fc{-}$ is functorial.


  In either case, we have the span of span maps:
  \begin{equation}\label{eq:2morphismspanspan}
    \xymatrix{
      & G \wquot G \ar[dl]_{id} \ar[dr]^{id} & \\
      G \wquot G & (G^2) \wquot G  \ar[r]^{m} \ar[l]_{m} \ar[d]_{i_3} \ar[u]^{m} & G \wquot G  \\
      & (G^3) \wquot G \ar[ul]^{s}\ar[ur]_{t} & \\
    }
  \end{equation}
  where the maps are given by:
  \begin{equation}
    i_3 : (g_1,g_2) \mapsto (g_1,g_2,1)
  \end{equation}
  and
  \begin{equation}
    s: (g_1,g_2,g_3) \mapsto g_2^{-1} g_1
  \end{equation}
  and
  \begin{equation}
    t : (g_1,g_2,g_3) \mapsto g_2^{-1} g_3 g_1 g_3^{-1}
  \end{equation}

  We can then apply $\Lambda$ to find $Z_G(M)$. Again, by
  functoriality, we can either compose the 2-linear map for the pair
  of pants $Y$ and its dual, or apply $\Lambda$ to
  (\ref{eq:2morphismspanspan}). Following \cite{gpd2vect}, the natural
  transformation $Z_G(M) : Z_G(S) \ra Z_G(S')$ is in general:
  \begin{equation}
    Z_G(M) = \epsilon_{L,i_3} \circ N \circ \eta_{R,m}
  \end{equation}
  For a connection $a$ on the source and representation $F$ of $G
  \wquot G$ (the central groupoid of the top span), these maps do the
  following to a vector $v$ in the representation space $F(a)$:
  \begin{equation}
    \eta_{R,m}(F)(a) : v \mapsto \bigoplus_{[A]|m(y)\cong a} \frac{1}{\# Aut(y)} \sum_{h \in Aut(a)} h^{-1} \otimes h(v)
  \end{equation}
  and
  \begin{equation}
    \epsilon_{L,i_3}(F')(a') : \bigoplus_{[A] | i_3(A) \cong a'} h_A \otimes v \mapsto \sum_{[A] | i_3(A) \cong x} i_3(h_A)v
  \end{equation}
  (Objects $A$ are connections on $M$). So the composite gives:
  \begin{equation}
    Z_G(M)(F)(a)(v) = \bigoplus_{a' | i_3(A) \cong a',  m(A) \cong a} \frac{1}{\#Aut(A)} \sum_{[A]|i_3(A) \cong a'} \sum_{h \in Aut(A)} i_3(h^{-1})h(v)
  \end{equation}
  In particular, a connection $a \in G \wquot G$ is given by one group
  element $g \in G$, so this amounts to:
  \begin{equation}
    Z_G(M)(F)(g)(v) = \bigoplus_{g' \in G} \frac{1}{\# Aut(g',g'g^{-1})} \sum_{h \in Aut(g',g'g^{-1})} v
  \end{equation}
  Note that in this case, the only objects of $G^3 \wquot G$ with
  nontrivial contribution are those of the form $(g',g'g^{-1},1)$, and
  $i_3$ takes a gauge transformation represented by $h$ to one
  represented by the same $h$. Thus, $i_3(h^{-1})h(v) = v$.
\end{example}

\section{Twisting and 2-Linearization}\label{sec:twisting}

We now would like to see how our categorified quantization process
$\FV: \Span(\Gpd) \ra \iiV$ generalizes to a twisted 2-linearization,
$\FV^{U(1)}$.  Then our main result generalizes to a claim that the
twisted DW model is a composite of $\FV^{U(1)}$ and a ``classical
field theory'' valued in groupoids equipped with such cocycles.  These
classical field theories are then classified by choices $(G,\omega)$,
where $G$ is a finite group and $\omega \in H^3_{grp}(G,U(1)$ is a
cohomology class.  The ``quantization functor'' $\FV^{U(1)}$ will be
the same for all such ETQFTs.

Throughout this section, there will be several theorems which involve
the construction of \textit{symmetric monoidal bicategories}, and
symmetric monoidal 2-functors (elsewhere often called symmetric
monoidal homomorphisms) between them. We will not be completely
explicit about every aspect of these claims.

The literature on symmetric monoidal bicategories, and the 2-functors,
natural transformations, and modifications between them (which form a
tricategory) is somewhat scattered, and definitions are seldom given
in full explicit detail. One source relevant to the application to
TQFT is Chris Schommer-Pries' Ph.D. thesis \cite{csp-class-2d}, which
collects together all the basic definitions in one place, though even
this relies on coherence diagrams which must be found in works by
Gordon, Power, and Street \cite{GPS-tricat}, and by McCrudden
\cite{mccrudden}. A recent work by Stay \cite{stay-ccb} gives a
detailed definition and proves that a certain bicategory of spans,
related to but distinct from the one we use here, is a symmetric
monoidal bicategory. It is also worth noting that the definitions in
these works can be seen as a special case of the definition of a
tricategory, whose algebraic form is given by Gurski
\cite{gurski-tricat}.

Even to fully state these definitions is quite lengthy. In particular,
to fully define a symmetric monoidal bicategory requires a total of 13
pieces of data, including 1-morphisms such as associators, unitors,
and 2-morphisms such as that which takes the place of the pentagon
identity for associators in a monoidal category. These various cells,
in turn, must satisfy 7 coherence conditions expressed by the
commutation of various sizeable diagrams of 2-cells. This does not
even include the data implicit in the definition of a bicategory
without any defined monoidal structure. The definition of a symmetric
monoidal 2-functor likewise involves 7 pieces of data and 5 coherence
diagrams, beyond those defining a 2-functor (often called
``homomorphism'') between bicategories.

In what follows, we will not give all this data, nor verify all these
conditions explicitly. The reason is that the monoidal structure for
our categories is more strict than the completely general case. In
many cases, the data in our particular constructions are either just
identities, or else are derived from universal properties.  In these
cases, the coherence properties of such isomorphisms are automatically
satisfied because of such universal properties. Rather, the crucial
issue will be verifying that we have bicategories and 2-functors at
all, and this will occupy most of our attention.

\subsection{Cocycle Twisting as Homotopy QFT}\label{sec:hqftcocycle}

The above remarks imply that the action functional associated to
$\omega$ will be part of the classical theory, and $\FV^{U(1)}$ will
be the same 2-functor for all choices of $\omega$.  As we shall see,
the most natural way to do this uses a generalization of
$\Span(\Gpd)$, which we will call $\uispgpd$, in which groupoids may
come equipped with some $U(1)$-cocycle information, which $\FV^{U(1)}$
will respect.  This contains an isomorphic copy of $\Span(\Gpd)$ as
the sub-bicategory where this extra data is trivial - that is, all
cocycles take the constant value $1 \in U(1)$ - on which $\FV^{U(1)}$
restricts to $\FV$.

This cohomological aspect of the construction of TQFT from Lie groups
has been developed in detail by Freed, Hopkins, Lurie and Teleman
\cite{FHLT}. We remark that this framework focuses especially on
classifying TQFTs in dimension $n$ by means of an element of the
$n^{th}$ group cohomology of $G$, so since $\FV$ and $\FV^{U(1)}$ are
specifically 2-categorical our construction gives a version of this
theory which extends only to codimension-2.  The reconstruction of the
DW model when $n=3$ is of special interest since 1D manifolds are
sufficiently simple to be an interesting stopping point.

Recall that a group cohomology element is an element of the
ordinary cohomology of the classifying space:
\begin{equation}
  [\omega] \in H^3(BG,U(1))
\end{equation}
Its role is best understood in terms of the fact that $G$-connections
on a manifold $M$ correspond to homotopy classes of maps $A : M \ra
BG$ into the classifying space of $G$.  This applies to both manifolds
and cobordisms, so we may understand the role of $[\omega]$ in the
context of the Homotopy Quantum Field Theory of Turaev \cite{HQFT}.

An HQFT is rather like a TQFT, except that the source category
consists of manifolds and cobordisms \textit{equipped with maps} into
a target space $X$. Specifically, one has a category
$\catname{nCob}/X$, whose objects are $(n-1)$-manifolds equipped with
maps into target space $X$. The whose morphisms are cobordisms
equipped with maps to $X$ whose restriction to the boundary agrees
with the maps from the source and target objects. Then an HQFT in
Turaev's sense is a functor from $\catname{nCob}/X$ to
$\catname{Vect}$.

In the case where $X = BG$, this amounts to saying that an HQFT is an
assignment of vector spaces and linear maps to manifolds and
cobordisms equipped with a $G$-bundle with connection.  Such HQFTs are
classified by cohomology classes on $X$ ($BG$ in this case). An
equivalent way to say this is that they are classified by gerbes on
$X$ (since the cohomology class determines the gerbe). 

In a somewhat more specialized case, this has been described by Picken
\cite{picken}, as a ``rank-1 TQFT''. Picken's construction is given in
terms of gluing rules for manifolds with boundary, but is equivalent
to the categorical description in terms of composition of cobordisms.
Moreover, these are cobordisms are labelled with specific collections
of neighborhoods, so that the abstract cohomology class $[\omega]$ (of
the map into $BG$) is represented by concrete transition functions for
a gerbe, relative to these neighborhoods. Being rank-1 means that
these transition functions are valued in $U(1)$ - that is, that
$\omega$ is a $U(1)$-valued cocycle.

Picken (\cite{picken}, Theorem 4.6) proves that there is a 1-1
correspondence between such $U(1)$ gerbes, and a certain class of
rank-1, 2-dimensional TQFT's. In particular, the formulation given
there is helpful in understanding the composition in the new
bicategory $\uispgpd$ introduced in Section \ref{sec:uispgpd}, since
the assignments made using $\omega$ will satisfy certain gluing rules
which amount to the characterization of these TQFT's as functors.

In the untwisted situation, the classical field theory $\fc{-}$
assigns a groupoid $\fc{S}$ of all connections, which are given by
homotopy classes of maps into $BG$ (connections) for each manifold
$S$.  The exact correspondence is that the mapping space $Maps(S,BG)$
is the classifying space for this groupoid.  The classical field
theory when there is a cocycle $\omega$, which we call
$\fc{-}^{\omega}$, will produce the same groupoid. However, it also
associates cocycle information to that groupoid, which is derived from
$\omega$, as described below.  Then, where $\FV$ simply performs a sum
(or direct sum) over all the objects of this groupoid in a span, the
twisted form $\FV^{U(1)}$ adjusts these sums using this cocycle
information.

Now we describe the category of groupoids with cocycle which we need
to make this work.

\subsection{The Symmetric Monoidal Bicategory $\uispgpd$}\label{sec:uispgpd}

In a previous work \cite{catalgQM}, the author described a monoidal
category of groupoids with phases valued in $U(1)$, a special case of
groupoids with labels valued in a monoid $M$.  The motivation there
was to allow for a more physically realistic model of the quantum
harmonic oscillator in a category of groupoids and spans.  The
$U(1)$-phases were needed to get interesting time evolution operators.
This involved spans of groupoids equipped with phases derived from a
``number-operator'', which plays the role of a \textit{Hamiltonian}
for that system.  Now we want to describe a variation on this,
thinking of phases in the \textit{Lagrangian} sense, as
(exponentiated) actions rather than energies, but an analogous
structure is required.

We wish to extend our factorization $\FV \circ \fc{-} = Z_G(-)$ to
include the twisted case, where a nontrivial Lagrangian is present.
So there should be a factorization through a bicategory which contains
$\Span(\Gpd)$.  Our bicategory $\uispgpd$ will look like
$\Span(\Gpd)$, except that groupoids come equipped with some extra
data.

The simplest such data, most obviously related to Lagrangians in the
standard physical sense, is the assignment of an (exponentiated)
action in $U(1)$ to a history for a system.  That is, the
$U(1)$-element is assigned to an object of the groupoid in the middle
of a 2-morphism.  This function ought to be an invariant of
isomorphism classes of objects (physically indistinguishable histories
get the same action).  When we apply $\FV$ to a 2-morphism, we sum
over such ``histories''.  In \cite{catalgQM} it is explained how
groupoidification can be extended to replace groupoid cardinality with
a weighted sum:
\begin{equation}
  |(\mathcal{G},f)| = \sum_{[x] \in \underline{\mathcal{G}}} \frac{f(x)}{|\opname{Aut}(x)|}
\end{equation}
This is naturally found in $\mathbb{R}^+ \otimes U(1)$, which we map
into $\mathbb{C}$ (identifying all elements $(0,\phi)$ with $0 \in
\mathbb{C}$).  

This allowed the construction of a full complex Hilbert space in
\cite{catalgQM}.  Now, the role of cardinality in groupoidification arises
from the Nakayama isomorphism between the left and right adjoints of
the restriction functors.  In $Hom(\catname{1},\catname{1})$, as
discussed in \cite{gpd2vect}, this isomorphism simply becomes a
numerical factor, the groupoid cardinality.  Thus, as might be
expected, our $\FV^{U(1)}$ will incorporate the $U(1)$-valued
topological action into a twisting of the Nakayama isomorphism, at the
2-morphism level.

In building $\uispgpd$ as a monoidal bicategory, it is not sufficient
simply to take spans in $U(1)-\Gpd$: that is, groupoids with
$U(1)$-functions on them.  Instead, we need a different structure to
describe the appropriate ``categorification of the action
functional'', which will reproduce the twisted DW model.  The key
point is that $U(1)$ phases on objects can be understood as 0-cocycles
in groupoid cohomology.

With our overall aim in mind, we will define a bicategory in which
this classical process takes values:

\begin{definition}The symmetric monoidal bicategory $\uispgpd$ has:
  \begin{itemize}
  \item \textbf{Objects}: (essentially finite) groupoids $A$ equipped
    with 2-cocycles $\theta \in Z^2(A,U(1)$
  \item \textbf{1-Morphisms}: a morphism from $(A,\theta_A)$ to
    $(B,\theta_B)$ is a span of groupoids $A \stackrel{s}{\la} X
    \stackrel{t}{\ra} B$, equipped with a 1-cocycle $\alpha \in
    Z^1(X,U(1))$
  \item \textbf{2-morphisms}: a 2-morphism from $(X,\alpha,s,t)$ to
    $(X',\alpha ',s',t')$ in $Hom((A,\theta_A),(B,\theta_B))$ is an
    equivalence class of spans of span maps $X \la Y \ra X'$ equipped with
    0-cocycle $\beta \in Z^0(Y,U(1))$.

    (The equivalence is taken up to a $\beta$-preserving groupoid
    isomorphism, $(Y,\beta_Y) \rightarrow (Y',\beta_{Y'})$, which
    commutes with the source and target maps)
  \end{itemize}

  This data must satisfy the following conditions:
  \begin{itemize}
  \item In any 1-morphism
    \begin{equation}
      (X,\alpha,s,t) : (A,\theta_A) \ra (B,\theta_B)
    \end{equation}
    the cocycles satisfy
    \begin{equation}\label{eq:iicocycondition}
      s^{\ast}\theta_A = t^{\ast}\theta_B
    \end{equation}
    In particular, $[s^{\ast}\theta_A] = [t^{\ast}\theta_B]$
  \item In any 2-morphism
    \begin{equation}
      (Y,\beta,\sigma,\tau) : (X_1,\alpha_1,s_1,t_1) \Rightarrow (X_2,\alpha_2,s_2,t_2)
    \end{equation}
    the cocycles satisfy
    \begin{equation}\label{eq:icocycondition}
      (\sigma^{\ast} \alpha_1) (\tau^{\ast} \alpha_2)^{-1} = 1
    \end{equation}
    (In particular, $[\sigma^{\ast}\alpha_1] = [\tau^{\ast}\alpha_2]$,
    but moreover, since a 0-cocycle on a groupoid is an
    invariant function, $\delta \beta = 1 \in U(1)$ and the cocycles
    themselves are equal.)
  \end{itemize}
  The structures making $\uispgpd$ a monoidal bicategory are:
  \begin{itemize}
  \item Composition of 1-morphisms
    \begin{equation}
      (X_1,\alpha_1,s_1,t_1) : (A,\theta_A) \ra (B, \theta_B)
    \end{equation} and
    \begin{equation}
      (X_2,\alpha_2,s_2,t_2) : (B,\theta_B) \ra (C,\theta_C)
    \end{equation} at the object $(B,\theta_B)$ gives the
      same span of groupoids as in $\Span(\Gpd)$, and assigns the
      pullback object the cocycle
    \begin{equation}\label{eq:1mor-twisted-comp}
      \alpha_1 \cdot \alpha_2 \cdot \theta_B
    \end{equation} (explained below)
  \item Vertical composition of 2-morphisms:
    \begin{equation}
      (Y,\beta,\sigma,\tau) : (X_1, \alpha_1,s_1,t_1) \ra  (X_2, \alpha_2,s_2,t_2)
    \end{equation}
    and 
    \begin{equation}
      (Y',\beta ',\sigma ',\tau ') : (X_2, \alpha_2,s_2,t_2) \ra  (X_3, \alpha_3,s_3,t_3)
    \end{equation}
    at $(X_2,\alpha_2)$ gives the same groupoids as in $\Span(\Gpd)$,
    with the cocycle given by
    \begin{equation}\label{eq:2mor-twisted-comp}
      \beta \cdot \beta ' \cdot \alpha_2
    \end{equation}
  \item Horizontal composition of 2-morphisms:
    \begin{equation}
      (Y,\beta,\sigma,\tau) : (X_1, \alpha_1,s_1,t_1) \ra  (X_2, \alpha_2,s_2,t_2)
    \end{equation}
    in $Hom ( (A,\theta_A),(B,\theta_B))$, and 
    \begin{equation}
      (Y',\beta ',\sigma ',\tau ') : (X_1', \alpha_1',s_1',t_1') \ra  (X_2', \alpha_2',s_2',t_2')
    \end{equation}
    in $Hom ( (B,\theta_B), (C,\theta_C))$ at $(B,\theta_B)$ gives the
    same groupoids as in $\Span(\Gpd)$, with the cocycle given by
    \begin{equation}\label{eq:2mor-twisted-horiz}
      \beta \cdot \beta ' \cdot \theta_B
    \end{equation}
  \item The monoidal structure is given by
    \begin{equation}\label{eq:twisted-monoidal}
      (A,\theta_A) \otimes (B,\theta_B) = (A \times B, \theta_A \cdot \theta_B)
    \end{equation}
    on objects, and on morphisms and 2-morphsims in the same way at
    each position in the span or span-of-spans diagram.
  \end{itemize}
\end{definition}

We explicitly define the cocycle $\alpha_1 \cdot \alpha_2 \cdot
\theta_B$ in (\ref{eq:1mor-twisted-comp}) as follows.  First recall
that spans of groupoids are composed by taking the weak pullback of
$t_1$ and $s_2$, which is the iso-comma groupoid $t_1 \downarrow s_2$.
Its objects are triples $(x_1,f,x_2)$ where $f : t_1(x_1) \ra s_2(x_2)
\in B$, and its morphisms are pairs $(g_1,g_2) \in X_1 \times X_2$,
forming commuting squares in $B$:
\begin{equation}\label{eq:pullback-morphism}
\xymatrix{
t_1(x_1) \ar[r]^{f} \ar[d]_{s_1(g_1)} & s_2(x_2) \ar[d]^{t_2(g_2)} \\
t_1(x'_1) \ar[r]_{f'} & s_2(x'_2)
}
\end{equation}

For 1-morphisms $(X,\alpha,s,t) : (A,\theta_A) \ra (B,\theta_B)$, the
defining property of the 1-cocycle $\alpha$ is that it determines
a functor $\alpha : X \rightarrow U(1)$, where $U(1)$ is understood as
a groupoid with one object.  On the other hand, the 2-cocycle
$\theta_B$ is a map from pairs of morphisms $(f,f') \in B$ to $U(1)$
satisfying the 2-cocycle property, which ensures that ``twisting''
multiplication by $\theta_B$ remains associative if it was so
originally.

Then the functor $\alpha_1 \cdot \alpha_2 \cdot \theta_B : (t_1
\downarrow s_2) \rightarrow U(1)$ assigns the 1-morphism
(\ref{eq:pullback-morphism}) the value:
\begin{equation}
  \alpha_1(g_1) \cdot \alpha_2(g_2) \cdot \theta_B(f,f')
\end{equation}
(This is the meaning of ``twisting multiplication by $\theta_B$'' above).

The horizontal composition rule for 2-morphisms is similar, in that we
compose by weak pullback over $(B,\theta_B)$ for both 1-morphisms and
2-morphisms.  All the twistings by $\theta_B$ are compatible under the
maps of the 2-morphism's inner span.

The vertical composition rule for 2-morphisms is also similar,
except that the 0-cocycle is assigned to objects of $Y' \circ Y$.
These are again of the form $(y', f, y)$ where $f : \sigma ' (y') \ra
\tau(y) \in X_2$.  So the composite cocycle assigns this object the
value $\beta '(y') \cdot \alpha_2(f) \cdot \beta(y)$.

This composition rule may be surprising at first sight, in that one
might naively expect the cocycle on $Y' \circ Y$ to be $\beta' \cdot
\beta$, or on $X_2 \circ X_1$ to be $\alpha_2 \cdot \alpha_1$, so that
a 0-cocycle is a product only of 0-cocycles, and so on.  It is clear,
though, that objects in the $\tau \downarrow \sigma '$ contain a
morphism from $X_2$, or that $t_1 \downarrow s_2$ contains two
morphisms from $B$, so this rule is a consequence of the use
of the \textit{weak} pullback to compose spans of groupoids.  We will
describe this in terms of transition functions in the case of
groupoids of connections in Section \ref{sec:twistedfc}.  For the
moment, the intuitive idea is that in the composite of two spans of
groupoids, objects in the $Y$ or the $X_i$ need not match exactly as
in a fibered product, and some ``twisting'' is needed to align them
correctly.  As we will see, this is necessary to ensure that the
composition rules indeed determine cocycles on the groupoids found by
such weak pullback.

Of course, we must check that this structure indeed defines a
symmetric monoidal bicategory. To see this, it is useful to recall why
$\Span(\Gpd)$ is one.

\begin{lemma}
$\Span(\Gpd)$ is a symmetric monoidal bicategory.
\end{lemma}
\begin{proof}
  The main principle used is the fact that composition of spans is
  given by weak pullback, is defined by a universal property. The
  unique canonical maps from this property gives provides associator
  and unitor 2-cells for composition, which we do not name
  explicitly. They arise from the canonical maps from a weak pullback,
  expressed as spans with one identity leg. The fact that they are
  determined by a universal property ensures that they satisfy the
  necessary coherence conditions. For example, the pentagon identity
  asserts that two composites are equal, which in this case are both
  the same canonical map from the universal property of the pullback.

  The monoidal structure on $\Span(\Gpd)$, denoted by $\otimes$, is
  given on objects by the Cartesian product in $\Gpd$ (which is not a
  categorical product in $\Span(\Gpd)$). This extends naturally to
  spans. Again, a monoidal bicategory must have associator and unitor
  cells for the monoidal product. We take these to be precisely the
  canonical maps which come from the universal property of the product
  in $\Gpd$, interpreted as morphisms in $\Span(\Gpd)$ with one leg
  the identity map. For this reason, we do not explicitly name them.

  The definition of a monoidal bicategory then demands the existence
  of various 2-cells which take the place of the pentagon identity,
  unitor identites, and other equations which hold in monoidal
  categories for composites of associators and unitors. There is a
  slight complication in defining these, since the composite of
  1-morphisms, by weak pullback, is only defined up to an isomorphism
  2-cell.

  However, since the associator and unitor 1-morphisms come from maps
  in $\Gpd$, we may adopt the convention to choose the composite of
  the spans which comes directly from the composite of the corresponding
  maps. In this case, the pentagon and other coherence 2-cells in
  $\Span(\Gpd)$ are just the identities, since they compare two
  composites which give canonical isomorphisms that come from the
  universal property of the product. These coherence 2-cells will
  therefore automatically satisfy the axioms for a symmetric monoidal
  bicategory, which each relate two identity diagrams. (Different
  choices of composites will give 2-cells which differ from these by
  the isomorphisms to our canonical choice of composite.)

  The symmetry data for $\Span(\Gpd)$ is trivial in a similar
  way. The braiding 1-morphism $\beta : A \otimes B \ra B \otimes A$
  is just the span which comes from the canonical morphism in $\Gpd$
  which comes from the universal property of the product.

  A symmetric monoidal bicategory has two hexagonal 2-cells relating
  composites of $\beta$ and the associator (taking the role of the
  hexagon identity for symmetric monoidal categories). Again, this
  structure is strict and we can choose the identity maps, provided we
  make the canonical choice of composites. For any other choice of
  weak pullback, these 2-isomorphisms are the canonical isomorphism to
  that canonical choice. Again, identities will automatically satisfy
  the coherence condition, and thus so will these other choices in
  case we make non-canonical choices for composites of 1-morphisms.

  Finally, the braiding map is strictly symmetric, so as above, given
  canonical choices for the corresponding spans, the structural
  2-isomorphism $\sigma : I \ra \beta \circ \beta$ is just the
  identity.
\end{proof}

Structures in $\uispgpd$ are exactly those in $\Span(\Gpd)$, together
with cocycle data. The previous lemma shows that all the properties
for a symmetric monoidal bicategory hold when the cocycle data is
trivial. It remains to check that the properties still hold when
\textit{nontrivial} cocycle data is added.

\begin{lemma}
$\uispgpd$ is a symmetric monoidal bicategory.
\end{lemma}
\begin{proof}
  We note that the underlying spans and spans-of-spans which define
  the data of the symmetric monoidal bicategory $\uispgpd$ will be
  exactly those which appear in $\Span(\Gpd)$. We need only define the
  associated cocycle data and see that the axioms still hold.

  Now, $\uispgpd$ is closed under composition of 1-morphisms if
  (\ref{eq:1mor-twisted-comp}) determines a 1-cocycle, namely
  $\alpha_1 \cdot \alpha_2 \cdot \theta_B$ is a functor from $(t_1
  \downarrow s_2)$ into $U(1)$.  Since $\alpha_1$ and $\alpha_2$ are
  functors, this follows precisely from the fact that $\theta_B$ is a
  2-cocycle, so that the twisted multiplication remains associative.
  In the same way, (\ref{eq:2mor-twisted-comp}) determines an
  invariant function on objects of $(\tau \downarrow \sigma ')$ (that
  is, a 0-cocycle), since $\alpha_2$ is a 1-cocycle, so its coboundary
  is trivial.

  Now, the structure morphisms (associator and unitors) for 1-morphism
  composition in $\uispgpd$ is just the same as those for
  $\Span(\Gpd)$, with the trivial cocycle (i.e. constant of value
  1). This is because the cocycle given by the two composition orders
  is preserved by the canonical isomorphism between the composites, so
  the identity (\ref{eq:icocycondition}) holds with the trivial
  cocycle. (A similar argument shows that composition of 2-morphisms
  is associative and has identities).
  
  Similarly, the fact that (\ref{eq:twisted-monoidal}) defines a
  monoidal product follows from the fact that the Cartesian product
  $\times$ on $\Gpd$ determines a monoidal product in $\Span(\Gpd)$,
  and multiplication is a monoid operation for $U(1)$.  The fact that
  this monoidal product is symmetric follows from the symmetry
  isomorphism from the universal property of $\times$ and
  commutativity of multiplication in $U(1)$.  The monoidal unit is
  $(\catname{1},1)$, as can easily be verified.
\end{proof}

An obvious but important fact is:

\begin{corollary}\label{cor:subcat}
  The symmetric monoidal bicategory $\uispgpd$ contains $\Span(\Gpd)$
  as a sub-symmetric monoidal bicategory.
\end{corollary}
\begin{proof}
  There is a fully faithful symmetric monoidal 2-functor embedding any
  object, morphism, or 2-morphism of $\Span(\Gpd)$ into $\uispgpd$
  taking any groupoid to the same groupoid equipped with the trivial
  cocycle which has the constant value 1, and leaving all maps in
  spans unchanged. This is clearly a functor, since all operations in
  $\uispgpd$ are just the same as those in $\Span(\Gpd)$ when cocycles
  are disregarded.  The image of this embedding is a sub-category since
  it contains all identities and the monoidal unit, and is closed
  under the composition and monoidal operations of $\uispgpd$.
\end{proof}

Knowing that $\uispgpd$ is a symmetric monoidal bicategory, we want to
construct the two symmetric monoidal 2-functors $\fc{-}^{\omega}$
(given a fixed 3-cocycle $\omega \in Z^3(BG,U(1))$), and $\FV^{U(1)}$.
We address these next.

\subsection{The Classical Field Theory}\label{sec:twistedfc}

Our construction of the ETQFT $Z_G^{\omega}$ corresponding to the
twisted DW model will use a generalization of the factorization
$Z_G(-) = \FV \circ \fc{-}$.  The quantization functor $\FV^{U(1)} :
\uispgpd \ra \iiV$ will always be the same, and the cocycle $\omega$
will modify only the classical field theory, by defining the cocycle
data in $\uispgpd$.

We will begin by describing the classical field theory component.

The "topological action" for the twisted DW theory comes from a
$U(1)$-valued class in group cohomology:
\begin{equation}
  [\omega] \in H^3_{gp} (G,U(1))
\end{equation}
which we can take as represented by some particular cocycle:
\begin{equation}
  \omega \in Z^3_{gp} (G,U(1))
\end{equation}
Now, group cohomology is just the usual third cohomology of the
classifying space of G, so in fact this says:
\begin{equation}
  \omega \in Z^3(BG, U(1))
\end{equation}
This is a function which, given a 3-cycle in the space $BG$, defines a
number in $U(1)$, satisfying the cocycle condition.  It is usual to
think of the classifying space defined simplicially, and therefore to
consider what $\omega$ does to 3-simplices.

Now, the classical part of the DW construction with associated
3-cocycle $\omega$ gives cocycles of different degree associated to
the groupoids of connections for manifolds of different dimension.  In
general, a $k$-dimensional cobordism will produce a groupoid (as
object, or part of a span) which has a $(3-k)$-cocycle associated to
it.  This data arises from the ``transgression'' of the original
cocycle, an algebraic structure explained very nicely by Willerton
\cite{willerton}.  We briefly summarize it here.

First, this concept uses the fact that a (flat) connection on a
$k$-dimensional manifold $M$ can be understood as a homotopy class
of maps into the classifying space, $f : M \ra BG$.  Recall that an
indirect definition of $BG$ is precisely this fact.  The classifying
space functor $B$ is right adjoint to the fundamental-groupoid functor
$\Pi_1$, so that in particular $Hom(\Pi_1(M),G) \cong Hom(M,BG)$,
where the second term consists of homotopy classes of maps of spaces.
The most important feature of $BG$ is that its fundamental group is
$G$ and all other homotopy groups are trivial.  More concrete
constructions of $BG$ for particular $G$ depend on exactly which
category of spaces $BG$ is considered to lie in.

One standard choice, used in the ``bar construction'' is that $BG$ is
a simplicial set.  It is constructed by taking a single base-point (if
$G$ is a group, or one base-point for each object if $G$ is a
groupoid), adding edges for each element of $G$, and then adjoining
higher-dimensional cells as necessary to make sure there are no higher
homotopy groups.  For instance, one would add: a triangular face
adjoining edges $f$, $g$ and $fg$ to make this loop contractible; a
tetrahedron between four such triangles expressing each associativity
relation; and so on.  For convenience, we take $BG$ as presented as
such a simplicial complex.  A 3-cocycle $\omega$ on $BG$ then gives a
value in $U(1)$ to each 3-simplex $\Delta_3$ in $BG$.  (A smooth
realization of $BG$ will treat this as an integral of some 3-form over
the 3-chain $\Delta_3$.  In general, when describing integration in
the group $U(1)$, we will treat it as the additive group
$\mathbb{R}/\mathbb{Z}$).

Now, a connection is given by a point in the space $Maps(M,BG)$.  This
space in turn is a simplicial complex, and in fact is the classifying
space of the groupoid of flat connections on $M$ (each isomorphism
class of objects corresponds to a connected component of this space
given by one of the base-points).  Since a point in $Maps(M,BG)$ is a
function, there is the evaluation map:
\begin{equation}
ev: M \times Maps(M,BG) \ra BG
\end{equation}
The image of $f : M \ra BG$, or rather of $M \times {f}$, under $ev$
is then a $k$-chain (perhaps degenerate) in $BG$, namely the image
$f(M)$.  If we take a $(3-k)$-simplex $\Delta_{3-k}$ in $Maps(M,BG)$,
then these images in $BG$ form a 3-dimensional subspace which looks
like $M \times \Delta_{3-k}$.  This can be decomposed into individual
simplices in $BG$.

But then, this means we have a $(3-k)$-cocycle on $Maps(M,BG)$, the
``transgression'' of $\omega$ to $Maps(M,BG)$, which is denoted:
\begin{equation}
\tau_M(\omega) \in H^{3-k}(Maps(M,BG),U(1))
\end{equation}
It is given by integrating $\omega$:
\begin{equation}\label{eq:transgression-defn}
\tau_M(\omega) = \int_M ev^{\ast}(\omega)
\end{equation}
It can be integrated over $\Delta_{3-k}$ to get an element of
$U(1)$. Our classical field theory will assign the cocycle
$\tau_M(\omega)$ to each groupoid $\fc{M}$.

So for 2-morphisms in $\catname{3Cob_2}$, which are 3-dimensional
cobordisms $M$ ("spacetimes with boundary"), this just amounts to an
action functional: a $U(1)$-valued function on connections.  In the
untwisted case, we have the constant function $\omega \cong 1$, and
thus $\tau_M(\omega) \cong 1$ also.  But for objects (1-dimensional
manifolds) and morphisms (2-dimensional cobordisms), we get different
data: respectively, 2-cocycles and 1-cocycles.

\begin{definition}\label{def:twistedconn}
  For a fixed finite group $G$ and group 3-cocycle $\omega$, the
  classical field theory is a symmetric monoidal 2-functor:
  \begin{equation}
    \fc{-}^{\omega} : \catname{3Cob_2} \ra \uispgpd
  \end{equation}
  which acts as follows:
  \begin{itemize}
  \item Objects: $\fc{B}^{\omega} = (\fc{B},\tau_B(\omega))$
  \item Morphisms: $\fc{S:B_1 \ra B_2}^{\omega} = (\fc{S}, \tau_S(\omega),
    i_1^*, i_2^*)$ (where the $i_j$ are the inclusion maps of the
    $B_j$ into $S$).
  \item 2-Morphisms: $\fc{M : S \ra S'}^{\omega} = ( \fc{M} , \tau_M(\omega),
    i^*, (i')^*)$, where again $i$ and $i'$ are inclusion maps of $S$
    and $S'$ into $M$.
  \end{itemize}
\end{definition}

This definition implicitly makes the assertion that this is a
symmetric monoidal 2-functor.  The first thing to check is that it
exists at all.

\begin{lemma}
  The construction for $\fc{-}^{\omega}$ gives well-defined maps for
  objects, morphisms, and 2-morphisms into $\uispgpd$.
\end{lemma}
\begin{proof}
  We need to check that the image of $\fc{-}^{\omega}$ actually lies
  in $\uispgpd$.  It is well-known that transgression will yield
  cocycles (see e.g. \cite{willerton}), so we need to verify the
  conditions (\ref{eq:iicocycondition}) and (\ref{eq:icocycondition})
  for those cocycles.

  Suppose that $S : B \ra B'$ is a cobordism, so that $\partial S = B
  \sqcup B'$, and applying $\fc{-}^{\omega}$ we get the span:
  \begin{equation}
    \xymatrix{
      &  (\fc{S},\tau_S(\omega)) \ar[dl]_{s} \ar[dr]^{t} & \\
      (\fc{B},\tau_B(\omega)) & & (\fc{B},\tau_{B'}(\omega))
    }
  \end{equation}
  Then we want to verify that the cocycles are compatible, or in other
  words that $(s^{\ast}\theta_B) (t^{\ast}\theta_{B'})^{-1} = 1$.
  Restating this with the cocycle taking values in the additive group
  $\mathbb{R}/\mathbb{Z}$ (since we want to express the value in terms
  of an integral):
  \begin{equation}
    s^{\ast}\theta_B - t^{\ast}\theta_{B'} = 0
  \end{equation}

  But this is a 2-cocycle on $\fc{S}$ given by:
  \begin{eqnarray}
    &   & (s^{\ast}\theta_B) - (t^{\ast}\theta_{B'}) \\
    \nonumber & = & s^{\ast}(\int_B ev^{\ast}(\omega)) - t^{\ast}(\int_{B'} ev^{\ast}(\omega)) \\
    \nonumber & = & \pi^{\ast}(\int_{\partial S} ev^{\ast}(\omega)) \\
  \end{eqnarray}
  Here, we are using the fact that the orientation on $B$ and $B'$,
  the boundary components of $S$, are opposite, by convention and
  denoting by $\pi = s \otimes \bar{t}$ the projection map from
  $\fc{S}$ to $\fc{\partial S}$.  So this says the difference of
  $s^{\ast}\theta_B$ and $t^{\ast}\theta_{B'}$ is
  $\pi^{\ast}(\tau_{\partial S}(\omega))$.

  This is the pullback of a 2-cocycle on $Maps(\partial S, BG)$ to a
  2-cocycle on $Maps(S,BG)$.  Now suppose we evaluate it on a 2-chain
  $\Delta_2$ in $Maps(S,BG)$, which we take to be a 2-simplex.  (A
  similar proof would work for non-simplicial constructions of $BG$).
  Then:
  \begin{eqnarray}
    & \pi^{\ast}(\tau_{\partial S}(\omega)) [\Delta_2] \\
    \nonumber = & \int_{\partial S \times \Delta_2} ev^{\ast}(\omega) \\
    \nonumber = &  \int_{ev(\partial S \times \Delta_2)} (\omega)
  \end{eqnarray}
  This is an integral of $\omega$ on a 3-chain in $BG$ which is one
  part of:
  \begin{equation}
    \partial( S \times \Delta_2 ) = ( \partial S \times \Delta_2 ) \cup ( S \times \partial \Delta_2 )
  \end{equation}
  So if we evaluate on the whole 3-chain, we have:
  \begin{align}
    &  \int_{ev (\partial (S \times \Delta_2))} (\omega) \\
    \nonumber = & \int_{ev (( \partial S \times \Delta_2 ) \cup ( S \times \partial \Delta_2 ))} (\omega) \\
    \nonumber = & \int_{ev (( \partial S \times \Delta_2 )} (\omega) +
    \int_{ ev (( S \times \partial \Delta_2 ))} (\omega)
  \end{align}
  Now, since
  \begin{equation}
    \int_{ev(S \times \partial \Delta_2)} (\omega) = \tau_S(\omega) [\delta \Delta_2]
  \end{equation}
  which is the evaluation of the 1-cocycle $\tau_S(\omega)$ on a
  boundary, this part is equal to $0$.  However, by Stokes' theorem:
  \begin{equation}
    \int_{ev (\partial (S \times \Delta_2))} (\omega) = \int_{ev(S \times \Delta_2)} \delta (\omega)
  \end{equation}
  and since $\omega$ is a cocycle, this is again $0$.  Thus, we have
  \begin{equation}
    \pi^{\ast}(\tau_{\partial S}(\omega)) [\Delta_2] = 0
  \end{equation}
  so that finally $s^{\ast}\theta_B = t'^{\ast}\theta_{B'}$ as
  required.

  A similar argument with 1-cocycles and 0-cocycles holds at the level
  of 2-morphisms.

  Now we have that $\fc{-}^{\omega}$ gives well-defined maps of
  objects, morphisms, and 2-morphisms from $\catname{3Cob_2}$ into
  $\uispgpd$.
\end{proof}

The next lemma verifies that $\fc{-}^{\omega}$ is indeed a symmetric
monoidal 2-functor, but some preliminary remarks may be useful.

We know that $\fc{-} : \catname{3Cob_2} \ra \Span(\Gpd)$ is a
symmetric monoidal 2-functor, since it comes from taking maps into
$BG$. Functoriality follows because this is local, and turns (weak)
pushouts into (weak) pullbacks. To verify the analogous fact for
$\fc{-}^{\omega}$, we first need to check that the cocycle data makes
it well-defined.

As mentioned in Section \ref{sec:hqftcocycle} that the cocycle values
themselves associated to manifolds and cobordisms with connection, by
Picken's construction, define an HQFT.  The properties proved in
\cite{picken} amount to the fact that such an HQFT is a (symmetric)
monoidal functor into vector spaces (1-dimensional in this case) from
a category of manifolds and cobordisms which are equipped with a map
into a target space $X$, which in this case is $BG$.  Such HQFT are in
1-1 correspondence with gerbes on $BG$, which are determined by
cocycles such as $\omega$, which represent the curvature form for the
gerbe.

Given $\fc{-}^{\omega}$, the composition rule for $\uispgpd$ discussed
in the previous section gets a useful geometric interpretation.  In
the special case where 1-morphisms are manifolds without boundary,
seen as cobordisms from the empty manifold to itself, the cocycle data
for any groupoid of connections may be seen as a by a rank-1 embedded
2-dimensional TQFT with target $BG$ in the sense of Picken (definition
4.1 of \cite{picken}).  This may be understood as a (unitary) TQFT (or
rather, HQFT, since manifolds and cobordisms are equipped with maps to
$BG$) in which every vector space is just the 1-dimensional vector
space $\mathbb{C}$.  (Note that the adjective ``unitary'' is
unnecessary for finite groups, since every element has finite order so
all values are in fact roots of unity.)

Thus, one gets an element of $U(1)$ for each morphism of objects
(i.e. between 1-manifolds equipped with connection - in other words,
for a gauge transformation), by what Picken calls $Z'$ (part of $Z$ in
our terminology), and an element of $U(1)$ for each cobordism by
Picken's $Z$.  Then the composition rule for cobordisms is Picken's
gluing rule, which agrees with our composition rule
(\ref{eq:2mor-twisted-comp}) in that case.  The point here is that one
gets an extra contribution from the boundary $B$ where two manifolds
are being glued.  This is explained in \cite{picken} in terms of
transition functions (for a gerbe induced from the gerbe on $BG$
classified by $\omega$).  Essentially, one must make a gauge
transformation to ensure that connections on the two cobordisms being
glued actually match at the boundary.  We may also understand it by
thinking of the gauge transformation identifying the different
connections on $B$ as a mapping cylinder: two copies of $B$ with
connections in different gauge, are identified with the ends of a
cylinder $B \times I$.  This has a nontrivial connection where the
holonomies along the edge $b \times I$ for each point $b \in B$
conjugates an holonomy for a loop based at that point.

This is the extra morphism in the weak pullback contributing to this
composition, as described in Section \ref{sec:uispgpd}.  

This is the idea behind the proof of the following.

\begin{theorem}\label{thm:fc-sm2func}
  The construction $\fc{-}^{\omega}$ gives a symmetric monoidal
  2-functor.
\end{theorem}
\begin{proof}
  First, we check that composition of 1-morphisms is preserved up to
  isomorphism.  Suppose $S : B_1 \ra B_2$ and $S' : B_2 \ra B_3$ are
  cobordisms.  Then
  \begin{equation}
    \fc{S' \circ S}^{\omega} =  (\fc{S' \circ S}, \tau_{S' \circ S}(\omega), i_1^*, i_3^*)
  \end{equation}
  On the other hand, since we have $\fc{B_2}^{\omega} = (\fc{B} ,
  \theta_{B_2})$ and $\theta_{B_2} = \tau_{B_2}(\omega)$, it follows
  from (\ref{eq:1mor-twisted-comp}) that:
  \begin{eqnarray}
    & \fc{S'}^{\omega} \circ \fc{S}^{\omega} \\ 
    \nonumber = & (\fc{S'} , \tau_{S'}(\omega), i_1^*, i_2^*)) \circ ( \fc{S}, \cdot \tau_S(\omega), i_2^*, i_3^*) \\
    \nonumber  \cong  & ( \fc{S' \circ S} , \tau_{S'}(\omega) \cdot \tau_S(\omega) \cdot \tau_{B_2}(\omega), I_1^*, I_3^* )
  \end{eqnarray}
  The composition of spans is just the weak pullback over $\fc{B_2}$.
  The $I_j$ are the inclusion maps of boundaries into the composite
  cobordism.

  Since we know $\fc{S' \circ S} \cong \fc{S'} \circ \fc{S}$, it
  suffices to check that
  \begin{equation}
    \tau_{S' \circ S}(\omega) = \tau_{S'}(\omega) \cdot \tau_S(\omega) \cdot \tau_{B_2}(\omega)
  \end{equation}
  under this identification. This is a 1-cocycle on the groupoid
  $\fc{S' \circ S}$ of connections on the composite.  That is, a
  $U(1)$-valued function on gauge transformations which respects their
  composition.

  Now, we are identifying the groupoid of connections on $S' \circ S$
  with the weak pullback of $\fc{S'}$ and $\fc{S}$ over $\fc{B_2}$ (to
  which it is naturally equivalent).  This means a connection on the
  whole space is determined by a pair of connections in $\fc{S'}
  \times \fc{S}$ identified by a gauge transformation between the
  restrictions of the connections to $B_2$.  (That is, there is a
  ``transition function'' specifying the change of gauge when gluing
  the connections at $B_2$).  A gauge transformation between two such
  objects is then a square of the form (\ref{eq:pullback-morphism}),
  and includes two gauge transformations from $\fc{B_2}$ - the
  transition function for the gauge transformations.  As in the
  discussion of Picken's HQFT above, this is assigned an element of
  $U(1)$, just as if we glued using a mapping cylinder $B_2 \times I$
  with a nontrivial connection.  This factor is precisely
  $\tau_{B_2}(\omega)$ by this construction.  The cocycle on $\fc{S'
    \circ S}$ is exactly this cocycle (pulled back through the
  equivalence with $\fc{S'} \circ \fc{S}$).

  A similar argument for 0-cocycles and gluing along 1-cocycles shows
  the composition of 2-morphisms is respected.  Together these imply
  the preservation of all identity 1- and 2-morphisms.

  It is straightforward to verify that there is a canonical
  isomorphism $\fc{A \sqcup B}^{\omega} \cong \fc{A}^{\omega} \otimes
  \fc{B}^{\omega}$ (often called $H$, as in \cite{csp-class-2d}).
  This is because the monoidal product in $\Span(\Gpd)$ is just the
  Cartesian product from $\Gpd$, and the same canonical isomorphism,
  seen as a span, will work here. Moreover, in the product, the
  cocycles simply multiply in $U(1)$.  On the other hand, the
  transgressed cocycles are
  \begin{eqnarray}
    \tau_{A \sqcup B}(\omega) = & \int_{A \sqcup B} ev^{\ast}(\omega) \\
    \nonumber  = & \int_A ev^{\ast}(\omega) + \int_B ev^{\ast}(\omega) \\
    \nonumber = & \tau_A(\omega) + \tau_B(\omega)
  \end{eqnarray}
  Since this sum is in $\mathbb{R}/\mathbb{Z}$, this is exactly what
  we expect.  Likewise, the monoidal unit is preserved up to a
  canonical isomorphism.

  These structural isomorphisms naturally satisfy all the coherence
  conditions for a functor of monoidal bicategories (as in
  \cite{GPS-tricat}) by the universal property, and the fact that the
  cocycles multiply strictly. A similar argument holds for the
  invertible modification relating the monoidal structure map $H$ and
  the symmetry map.

  Thus, $\fc{-}^{\omega}$ is a symmetric monoidal 2-functor.
\end{proof}

\subsection{The Twisted 2-Linearization Functor $\FV^{U(1)}$}

We have suggested that the twisted classical theory behind the DW
model takes values in $\uispgpd$.  We now want to understand the
twisted analog of $\FV$, the 2-linearization, or ``quantization''
functor $\FV^{U(1)}$.

The essential point is that we use the representations of the twisted
groupoid algebras such as $\mathbb{C}^{\theta_A}[A]$.  This is the
algebra of complex functions on morphisms of $A$, with the ``twisted''
multiplication:
\begin{equation}\label{eq:twistedproduct}
(F \star_A G)(f) = \sum_{g} F(g) G(g^{-1}f) \theta_A(g,g^{-1}f)
\end{equation}
The sum is taken over all morphisms $g \in A$ whose target is the
source of $f$: this is a twisted form of the usual convolution
product.  Given this notation, it is a standard fact that
representations of this twisted algebra correspond to so-called
``twisted representations'' $\rho$ of the groupoid itself, in which
the usual composition rule is replaced by $\rho(g_1) \circ \rho(g_2) =
\theta_A (g_1,g_2) \rho(g_1 \circ g_2)$.  It is also possible to
describe these as representations of a central extension of the
groupoid (more usual in the case of a group).  We will choose the
first of these descriptions for convenience.

So 2-cocycles on objects twist the representation categories that
appear as the output of $\FV$.  The 1-cocycles, as we will see, twist
the functors between them associated to spans, and the 0-cocycles
twist the natural transformations.

In particular, 1-morphisms will involve restriction and induction of
these twisted representations, for example pulling back along $s :
(X,\alpha_X) \ra (A,\theta_A)$ turns a representation of
$\mathbb{C}^{\theta_A}[A]$ into a representation of
$\mathbb{C}^{s^{\ast}\theta_A}[X]$.  By preceding arguments, this is the
same as $\mathbb{C}^{t^{\ast}\theta_B}[X]$ since $s^{\ast}\theta_A =
t^{\ast}\theta_B$).

It is possible to repeat what $\FV$ does to a span
$(X,s,t) : A \ra B$, which simply takes $t_{\ast} \circ s^{\ast}$:
pull back a representation to $X$ and push forward to $B$.  However,
if a cocycle $\alpha_X$ is present, we can ``twist'' this
identification of $\mathbb{C}^{s^{\ast}\theta_A}[X]$ with
$\mathbb{C}^{t^{\ast}\theta_B}[X]$ by $\alpha_X$.   This uses the map:
\begin{equation}
  M_{\alpha} : \mathbb{C}^{s^{\ast}\theta_A}[X] \ra \mathbb{C}^{t^{\ast}\theta_B}[X]
\end{equation}
which takes $f : X \ra \mathbb{C}$ to $\alpha \cdot f : X \ra
\mathbb{C}$.  Of course, this is actually an automorphism of one
algebra, since the twisting cocycles are actually equal by the
condition (\ref{eq:iicocycondition}).  It is convenient, however, to
represent $M_{\alpha}$ this way in what follows.  A well-known but
still crucial fact which we demonstrate here, is:

\begin{proposition}
 $M_{\alpha}$ is an algebra isomorphism.
\end{proposition}
\begin{proof}
  Clearly $M_{\alpha}$ is linear, so we check compatibility with the
  products.  Suppose $F,G : X \ra \mathbb{C}$, are thought of as
  elements of $\mathbb{C}^{s^{\ast}\theta_A}[X]$, with the product
  (\ref{eq:twistedproduct}).  Then applying $M_{\alpha}$, at $f \in X$
  we have:
  \begin{align}
    (M_{\alpha}(F) \star_B (M_{\alpha}(G)) (f) & =  \sum_{g} (\alpha \cdot F)(g) (\alpha \cdot G)(g^{-1}f) \theta_B(g,g^{-1}f) \\
    \nonumber    & = \sum_{g} \alpha(g) F(g) \alpha(g^{-1}f) G(g^{-1}f) \theta_B(g,g^{-1}f) \\
    \nonumber    & = \sum_{g} F(g) G(g^{-1}f) (\delta \alpha) (g,g^{-1}f)^{-1} \cdot \theta_A(g,g^{-1}f) \\
    \nonumber    & = \sum_{g} F(g) G(g^{-1}f) \alpha(f) \cdot \theta_A(g,g^{-1}f) \\
    \nonumber    & = \alpha(f) \cdot (F \star_A G) (f) \\
    \nonumber    & = (M_{\alpha}(F \star_A G))(f)
  \end{align}
  And indeed, since this map is plainly invertible with inverse
  $M_{\alpha^{-1}}$, this gives an isomorphism between the two
  algebras.
\end{proof}

This isomorphism induces a specific (contravariant) isomorphism
between the representation categories, by pre-composition:
\begin{equation}
M_{\alpha}^{\ast} : Rep(X,s^{\ast}\theta_A) \ra Rep(X,t^{\ast}\theta_B)
\end{equation}
We will use this in the construction for the 1-morphism map of
$\FV^{U(1)}$.

Just as 2-cocycles twist the objects (representation categories) and
1-cocycles twist the 1-morphisms (functors), so the 0-cocycles will
twist 2-morphisms (natural transformations).  The untwisted $\FV$ uses
the unit and counit for the adjunction between induction and
restriction functors of representations along groupoid homomorphisms.
There is still an adjunction for algebra representations, so this part
is much the same.  However, $\FV$ also uses the Nakayama isomorphism,
as in (\ref{eq:FV2mor-nakayama}).  This is a canonical choice of
isomorphism between the left and right adjoints to the restriction
functor.

However, since $\iiV$ is enriched in $\catname{Vect}_{\mathbb{C}}$, we
can of course ``twist'' this natural isomorphism by a scalar that
depends on a choice of object.  This is exactly the role of the
0-cocycle (which is physically interpreted as the complex-valued
``action'' for the configuration of our QFT that object represents).
\begin{definition}
  The ``twisted form'' of the Nakayama isomorphism:
\begin{equation}
  N_{\beta_Y} : \sigma_{\ast} \circ (M_{\sigma^{\ast}\alpha_1})^{\ast} \circ \sigma^{\ast} \Longrightarrow \tau_{\ast} \circ ( M_{\tau^{\ast}\alpha_2})^{\ast} \circ \tau^{\ast}
\end{equation}
which relates the ($\alpha$-twisted) forms of the left and right
adjunction acts at each object $y \in Y$ by:
\begin{equation}\label{eq:twistednakayama}
  N_{\beta_Y} : \bigoplus_{[y] | f(y) \cong x} \phi_y \mapsto \bigoplus_{[y] | f(y) \cong x} \frac{\beta_Y(y)}{\# Aut(y)}\sum_{g \in Aut(x)} g \otimes \phi_y(g^{-1})
\end{equation}
\end{definition}

This is just the same as the usual form, except for the factor of
$\beta_Y(y)$.  We note that this implicitly assumes that our spans of
span maps commute exactly - as in \cite{gpd2vect}, we might also need
to incorporate an explicit isomorphism up to which the diagram
commutes.  We note also that by (\ref{eq:icocycondition}), the maps
$(M_{\sigma^{\ast}\alpha_1})^{\ast}$ and
$(M_{\tau^{\ast}\alpha_2})^{\ast}$ are in fact equal, so again this
natural isomorphism is an automorphism.

Combining these twisted variants on the ingredients of $\FV$, we have
the following:

\begin{definition}\label{def:twistedlambda}
  The 2-functor
  \begin{equation}
    \Lambda^{U(1)} : \uispgpd \ra \iiV
  \end{equation}
  consists of the following assignments.
  \begin{itemize}
  \item \textbf{Objects}: $\Lambda^{U(1)}(A,\theta_A) = Rep(
    \mathbb{C}^{\theta_A}(A))$
  \item \textbf{Morphisms}: To a span $(X,\alpha_X,s,t) : (A,\theta_A)
    \ra (B,\theta_B)$ define a 2-linear map:
    \begin{equation}
      \FV^{U(1)}(X,\alpha_X,s,t) = t_{\ast} \circ (M_{\alpha_X})^{\ast} \circ s^{\ast}
    \end{equation}
    where $M_{\alpha_X} : \mathbb{C}^{s^{\ast} \theta_A}(X) \ra
    \mathbb{C}^{t^{\ast} \theta_B}(X)$ is the isomorphism of these
    groupoid algebras induced by multiplication by $\alpha_X$.
  \item \textbf{2-Morphisms}: to a 2-morphism $(Y,\beta_Y,\sigma,\tau)
    : (X_1,\alpha_1,s_1,t_1) \Rightarrow (X_2,\alpha_2,s_2,t_2)$
    assign the natural transformation:
    \begin{equation}\label{eq:fvui2mordef}
      \FV^{U(1)}(Y,\beta_Y,\sigma,\tau) = \epsilon_{L,\tau} \circ N_{\beta_Y} \circ \eta_{R,\sigma} : (t_1)_{\ast} \circ (M_{\alpha_1})^{\ast} \circ s_1^{\ast} \Longrightarrow (t_2)_{\ast} \circ (M_{\alpha_2})^{\ast} \circ s_2^{\ast}
    \end{equation}
  \end{itemize}
\end{definition}

\begin{remark}
  We have somewhat abused notation in order to write this in a
  balanced form.  Strictly speaking, we have that:
  \begin{equation}
    \eta_{R,\sigma} : \opname{Id}_{Rep(X_1,\alpha_1)} \Longrightarrow \sigma_{\ast} \circ \sigma{\ast}
  \end{equation}
  and similarly:
  \begin{equation}
    \epsilon_{L,\tau} :  \tau_{\ast} \circ \tau{\ast} \Longrightarrow \opname{Id}_{Rep(X_2,\alpha_2)}
  \end{equation}
  We have written them source and target, incorporating the
  multiplication operators $M_{\alpha_i}$ (and, though not written
  here, $M_{\sigma^{\ast}\alpha_1}$ and $M_{\tau^{\ast}\alpha_2}$).
  The point is just that
  \begin{equation}
    \sigma_{\ast} \circ (M_{\sigma^{\ast}\alpha_1})^{\ast} \circ \sigma^{\ast} \cong (M_{\alpha_1})^{\ast} \circ \sigma_{\ast} \circ \sigma^{\ast}
  \end{equation}
  and similarly for $\tau$.
\end{remark}

It may help to note that the cocycles at each level play somewhat
independent roles, in this construction, though with our specific
classical field theory $\fc{-}^{\omega}$ they are closely related via
transgression from $\omega$.  This close connection may be an
important part of the physical interpretation of this theory, and
ensures we have a functor from $\catname{3Cob_2}$, but it is not
essential to the ``quantization functor'' $\FV^{U(1)}$.  The
definition of $\FV^{U(1)}$ means we must have that
$s^{\ast}(\theta_A)$ and $t^{\ast}(\theta_B)$ differ by the coboundary
of $\alpha_X$ for the 2-linear map associated to a span to make sense
(see the proof of the Theorem \ref{thm:fvuifunctor} below).  However,
$\alpha_X$ must be a cocycle, hence has coboundary $0$.  So this is
simply the requirement that $s^{\ast}(\theta_A)t^{\ast}(\theta_B)^{-1}
= 1$ in the definition of $\uispgpd$.  A similar remark applies to the
2-morphisms.

This means that the deep underlying relation between the $\theta$,
$\alpha$, $\beta$ cocycles in our ETQFT is a property of the classical
field theory, not a requirement of the quantization functor.  Indeed,
part of the point of this factorization is that the quantization
functor contributes little to an understanding of the system: it
essentially looks at a specific representation in $\iiV$ of structures
already present in $\uispgpd$.  To say this is a ``representation'' is to
say precisely the following, which was implicitly stated in the above
definition:

\begin{theorem}\label{thm:fvuifunctor}
  The construction in Definition \ref{def:twistedlambda} determines a
  symmetric monoidal 2-functor
  \begin{equation}
    \FV^{U(1)} : \uispgpd \ra \iiV
  \end{equation}
\end{theorem}

This is the twisted version of (\cite{gpd2vect}, Thm. 5), though here
we are also explicitly noting that the 2-functor is symmetric
monoidal.  Much of the proof is substantially the same as the
untwisted case.  We need to check several facts, so we will prove it
as a series of lemmas, corresponding to lemmas and theorems shown in
the untwisted case in \cite{gpd2vect}.  The proofs are similar, so we
will cite those at the appropriate place for brevity where there is
significant overlap and show only the distinct new parts of the
proofs.

Moreover, as in Theorem \ref{thm:fc-sm2func}, in the interest of clarity
we have not defined all of the specified structure morphisms which
play the role for bicategories of properties of symmetric monoidal
functors between categories. Because the construction of the 2-functor
provides natural choices for these morphisms, shall remark on what
these choices are as we prove the necessary parts of this theorem.

To begin with, it is clear that $Rep( C^{\theta_A}(A) )$ is a 2-vector
space, since it is the category of representations of a finite
dimensional complex algebra on complex vector spaces.  Similarly, the
functorial constructions for 1- and 2-morphisms ensure that we must
obtain 2-linear maps and natural transformations.  We must show that
these assemble into a symmetric monoidal 2-functor.

\begin{lemma}\label{lemma:fvui1morcomp}
  $\FV^{U(1)}$ preserves composition of 1-morphisms up to isomorphism.
\end{lemma}
\begin{proof}
  Note that this is the twisted version of (\cite{gpd2vect}, Thm. 3),
  which gives the corresponding isomorphism for the composites of
  spans in $\Span(\Gpd)$. This gives one of the structure maps for a
  weak 2-functor. As we shall see, we can inherit this structure from
  the untwisted version.

  Suppose we are given two spans in $\uispgpd$:
  \begin{equation}
    (X_1,\alpha_1,s_1,t_1) :  (A,\theta_A) \ra (B,\theta_B)
  \end{equation}
  and
  \begin{equation}
    (X_2,\alpha_2,s_2,t_2) :  (B,\theta_B) \ra (C,\theta_C)
  \end{equation}
  Then the composite is:
  \begin{equation}
    (X_2 \circ X_1, \alpha_2 \cdot \theta_B \cdot \alpha_1,s_1 \circ S,t_2 \circ T)
  \end{equation}
  The cocycle is that given in (\ref{eq:1mor-twisted-comp}), and
  $(X_2 \circ X_1, S, T)$ are the groupoid and maps in the weak
  pullback of the cospan $(B,t_1,s_2)$.  It is shown in
  (\cite{gpd2vect}, Thm. 3) that there is an isomorphism
  \begin{equation}
    \gamma : T_{\ast} \circ S^{\ast} \ra (s_2)^{\ast} \circ (t_1)_{\ast}
  \end{equation}
  for the untwisted representation categories.  It suffices to show a
  similar natural isomorphism between two functors:
  \begin{equation}
    T_{\ast} \circ M_{\alpha_2 \cdot \theta_B \cdot \alpha_1} \circ S^{\ast} , (s_2)^{\ast} (t_1)_{\ast} : Rep^{s_1^{\ast}\theta_A}[X_1] \ra Rep^{s_2^{\ast}\theta_B}[X_2]
  \end{equation}
  
  First, note that the induction and restriction functors for twisted
  representations are given by the usual formulas for modules of
  rings, and that the twisted groupoid algebras are characterized as a
  direct sum of twisted group algebras.  These are the group algebras
  for automorphism groups of the objects in the $X_i$ and $B$, with
  multiplication twisted by the relevant 2-cocycles $\theta$.  For
  clarity, we will use the following notation for these algebras that
  appear in the restriction and induction formulas:
  \begin{align}
    \mathbb{A}_{x_1} = & \mathbb{C}^{s_1^{\ast}\theta_A}[Aut(x_1)] \\
    \mathbb{A}_{x_2} = & \mathbb{C}^{s_2^{\ast}\theta_B}[Aut(x_2)] \\
    \mathbb{A}_{t_1(x_1)} = & \mathbb{C}^{\theta_B}[Aut(t_1(x_1))] \\
    \mathbb{A}_{x_1,x_2} = & \mathbb{C}^{(s_1 \circ S)^{\ast} \theta_A} [Aut(x_1) \times_{Aut(t_1(x_1))} Aut(x_2)]
  \end{align}
  Note that the cocycles mentioned are necessarily equal to others -
  for instance, $s_1^{\ast} \theta_A = t_1^{\ast} \theta_B$, and so
  on.)

  So this natural transformation can be expressed at a stage $x_1 \in
  X_1$ in terms of its action on a representation $\rho$.  This is a
  linear map between spaces which are expressed as a direct sum over
  $x_2 \in X_2$, and in each such summand we have:
  \begin{equation}
    \gamma_{x_1}(F) : \mathbb{A}_{x_2} \otimes_{\mathbb{A}_{x_1,x_2}} \rho(x_1) \ra \mathbb{A}_{t_1(x_1)} \otimes_{\mathbb{A}_{x_1}} \rho(x_1)
  \end{equation}

  This is simply the twisted case of the usual formulas given as (91)
  and (92) in \cite{gpd2vect}.  The isomorphism given there as (94)
  will still work in the twisted case.  In the current notation, it
  acts in the following way.  The algebra $\mathbb{A}_{x_1,x_2}$
  decomposes as a direct sum over all $g \in Aut(t_1(x_1))$ (since it
  is a group algebra of a fibre product).  In the summand associated
  to $g$ we define the isomorphism to act on the generators of
  $\mathbb{A}_{x_2} \otimes_{\mathbb{A}_{x_1,x_2}} \rho(x_1)$ by:
  \begin{equation}
    (k \otimes v) \mapsto s_2(k) g^{-1} \otimes v
  \end{equation}
  which extends to the whole space.  This is still well defined,
  though now uses the twisted multiplication in the group algebra
  $\mathbb{A}_{x_2}$.  In particular, the underlying vector spaces are
  identical to the untwisted case.  The proof that this is an
  isomorphism is substantially the same as in the untwisted case.  The
  main difference is that the cocycle $\theta_B$ enters into the
  twisting of the multiplication in $\mathbb{A}_{x_1,x_2}$ and
  $\mathbb{A}_{t_1(x_1)}$.  However, this twisting is compatible with
  the isomorphism $\gamma$ and the algebra maps induced by the
  homomorphisms $t_1$ and $s_2$ into $(B,\theta_B)$, so the proof is
  the same.

  This then extends to an isomorphism between the two 2-linear maps
  \begin{align}
    & \FV^{U(1)}( (X_2,\alpha_2,s_2,t_2) \circ (X_1,\alpha_1,s_1,t_1) ) \\
    \nonumber = & (T \circ t_2)_{\ast} \circ (M_{\alpha_1 \cdot
      \theta_B \cdot \alpha 2})^{\ast} \circ (S \circ s_1)^{\ast}
  \end{align}
  and
  \begin{align}
    & \FV^{U(1)}(X_2,\alpha_2,s_2,t_2) \circ \FV^{U(1)}(X_1,\alpha_1,s_1,t_1) ) \\
    \nonumber = & (t_2)_{\ast} \circ (M_{\alpha_2})^{\ast} \circ
    (s_2)^{\ast} \circ (t_1)_{\ast} \circ (M_{\alpha_1})^{\ast} \circ
    (s_1)^{\ast}
  \end{align}
\end{proof}

Now we need the analogous fact for composition of 2-morphisms:

\begin{lemma}
  $\FV^{U(1)}$ preserves vertical composition of 2-morphisms strictly
  and horizontal composition of 2-morphisms up to the structure
  isomorphism of Lemma \ref{lemma:fvui1morcomp}.
\end{lemma}
\begin{proof}
  This is the twisted analog of (\cite{gpd2vect}, Lemma 4) and
  (\cite{gpd2vect}, Lemma 5).  The proofs are just the same except
  that we now have the factors $N_{\beta}$ in (\ref{eq:fvui2mordef}).
  It thus suffices that the $N_{\beta}$ are multiplicative under both
  horizontal and vertical composition of 2-morphisms.

  For vertical composition, suppose we are given 2-morphisms
  \begin{equation}
    (Y,\beta_Y,\sigma,\tau) : (X_1,\alpha_1,s_1,t_1) \Rightarrow (X_2,\alpha_2,s_2,t_2)
  \end{equation}
  and
  \begin{equation}
    (Y',\beta_{Y'},\sigma',\tau') : (X_2,\alpha_2,s_2,t_2) \Rightarrow (X_3,\alpha_3,s_3,t_3)
  \end{equation}
  (Note that the current notation is different from that of
  \cite{gpd2vect}, since here we use $(\sigma, \tau)$ instead of
  $(s,t)$, so that the structure maps for 2-morphisms are given by
  Greek letters and for 1-morphisms by Latin.)
 
  Then by (\ref{eq:2mor-twisted-comp}), we have: 
  \begin{align}\label{eq:2twisted-FV-comp}
    & \FV^{U(1)} ( (Y',\beta_{Y'},\sigma',\tau') \circ (Y,\beta_Y,\sigma,\tau) ) \\
    \nonumber  = & \FV^{U(1)} (Y' \circ Y, \beta_{Y'} \cdot \beta_Y \cdot \alpha_2, S \circ \sigma, T \circ \tau ') \\
    \nonumber = & \epsilon_{L,(T \circ \tau)} \circ N_{\beta_{Y'} \cdot \beta_Y \cdot \alpha_2} \circ \eta_{R,(S \circ \sigma)}
  \end{align}

  The terms appearing here are defined in and following Definition
  \ref{def:twistedlambda}.  On the other hand, we have:
  \begin{align}\label{eq:2twisted-comp-FV}
    & \FV^{U(1)} (Y',\beta_{Y'},\sigma',\tau') \circ \FV^{U(1)} (Y,\beta_Y,\sigma,\tau) ) \\
    \nonumber = & \epsilon_{L,(\tau ')} \circ N_{\beta_{Y'}} \circ
    \eta_{R,(\sigma ')} \circ \epsilon_{L,(\tau )} \circ N_{\beta_{Y}}
    \circ \eta_{R,(\sigma)}\\
  \end{align}
  This composite agrees with (\ref{eq:2twisted-FV-comp}) by a similar
  argument to that for 1-morphisms.  Namely, $S$ and $T$ are the maps
  for the weak pullback $(\sigma ' \downarrow \tau)$, and these maps
  are compatible with twisted multiplication.

  For horizontal composition, the proof is substantially the same as
  Lemma 5 of \cite{gpd2vect}, except that factors of $\beta$ appear in
  the sums, and twisting by $\theta_B$ makes the maps of the pullback
  compatible with the twisted multiplication.  The rest of the
  argument is substantially the same as for vertical composition.

  So we have that composition of 2-morphisms is preserved.
\end{proof}

\begin{lemma}
  The 2-functor $\FV^{U(1)}$ is naturally equipped with the structure
  of a symmetric monoidal 2-functor of bicategories.
\end{lemma}
\begin{proof}
  Here we must define some extra structure maps.  These arise
  naturally since tensor products essentially derive from $\times$ for
  groupoids.

First, (following
  \cite{gurski-tricat}) there should be an adjoint equivalences
  relating the monoidal units:
  \begin{equation}
    \iota : I \ra \FV^{U(1)}(I)
  \end{equation}
  relating the two monoidal units. Explicitly, this is:
  \begin{equation}
    \iota : \catname{Vect} \ra Rep(\mathbb{C})
  \end{equation}
  This is the obvious equivalence (indeed, isomorphism) of categories,
  since every vector space is automatically a representation of
  $\mathbb{C} = \mathbb{C}(1)$.

  Then we need natural adjoint equivalences relating monoidal products:
  \begin{equation}
    \chi{1,2} : \FV^{U(1)}(A_1,\theta_1) \otimes \FV^{U(1)}(A_2,\theta_2) \ra \FV^{U(1)}( (A_1,\theta_1) \otimes (A_2,\theta_2) )
  \end{equation}
  This comes from the natural inclusion:
  \begin{equation}
     \mathbb{C}^{\theta_1}[A_1] \otimes \mathbb{C}^{\theta_2}[A_2] \ra \mathbb{C}^{\theta_1 \cdot \theta_2}[A_1 \times A_2] 
  \end{equation}
  taking $f \otimes g$ to the function $fg(a,b) = f(a)g(b)$. It is
  easy to check this is an algebra homomorphism since $\otimes$ just
  multiplies cocycles, and the twisted multiplication acts
  independently in each factor. Furthermore, in finite dimensions,
  this is clearly an isomorphism, since any function is a finite
  linear combination of delta-functions supported on individual
  morphisms of the $A_i$, or of their product.
  
  The restriction and induction functors along this isomorphism are
  therefore an adjoint equivalence of categories.

  This means that we have:
  \begin{align}
    & \FV^{U(1)} ( (A_1, \theta_1) \otimes (A_2, \theta_2) ) \\
    \nonumber = & \FV^{U(1)} ( A_1 \times A_2 , \theta_1 \cdot \theta_2 ) \\
    \nonumber = & Rep^{\theta_1 \cdot \theta_2} [A_1 \times A_2] \\
    \nonumber = & Rep( \mathbb{C}^{\theta_1 \cdot \theta_2}[A_1 \times A_2]) \\
    \nonumber \cong & Rep( \mathbb{C}^{\theta_1}[A_1] \otimes \mathbb{C}^{\theta_2}[A_2] )
  \end{align}
  But this is generated by irreducible representations, and an
  irreducible representation of $\mathbb{C}^{\theta_1}[A_1] \otimes
  \mathbb{C}^{\theta_2}[A_2]$ is a tensor product of irreducible
  representations of $\mathbb{C}^{\theta_1}[A_1]$ and
  $\mathbb{C}^{\theta_2}[A_2]$.  So this is isomorphic to:
  \begin{align}
    & \FV^{U(1)} (A_1, \theta_1) \otimes \FV^{U(1)}(A_2,\theta_2) \\
    = & Rep^{\theta_1} [A_1] \otimes Rep^{\theta_2}[A_2]
  \end{align}
  with the tensor product the Deligne product of categories.  (For
  2-vector spaces, this looks just like the usual tensor product of
  vector spaces in terms of generators, up to isomorphism.)

  There must also be coherent invertible modifications which give
  isomorphisms of certain composites of these natural adjoint
  equivalences (see \cite{csp-class-2d}), and another relating
  $\chi$ and the braidings. These isomorphisms are quite natural, and
  a lengthy but straightforward check verifies the coherence
  conditions.
\end{proof}

Taking the above three lemmas together we have the proof of Theorem
\ref{thm:fvuifunctor}.  It is then immediate that this is an extension
of our original 2-linearization 2-functor to the larger category, in
the sense of the embedding noted in Corollary \ref{cor:subcat}.  That
is:

\begin{corollary}
  The restriction of $\FV^{U(1)}$ to $\Span(\Gpd) \subset \uispgpd$,
  is isomorphic with $\FV$.
\end{corollary}
\begin{proof}
  This is a straightforward consequence of applying the definitions
  with trivial cocycles, and the fact that the representation category
  of a finite groupoid is canonically isomorphic to that of its
  groupoid algebra.
\end{proof}

There is also a different special case, which is not immediately
relevant to our ETQFT context, but which we will point out since it is
immediate.  This extends the fact that $\FV$ restricted to
$Hom(\catname{1},\catname{1})$ reproduces Baez-Dolan groupoidification
(shown in \cite{gpd2vect}).  The new special case incorporates the
$U(1)$-groupoids of \cite{catalgQM}.  Then a $U(1)$-groupoid span (or
``stuff operator'' in the sense of \cite{catalgQM}) is simply a
nontrivial 2-morphism of $\uispgpd$ in
$Hom((\catname{1},1),(\catname{1},1))$ between two 1-morphisms with
trivial cocycles.  (Moreover, \cite{catalgQM} only considered the case
where the central objects of these spans are always the groupoid of
finite sets and bijections.  This is not essentially finite, as in the
present case, but provided we restrict to situations where all sums
converge, the same ideas apply.)

Finally, as in the untwisted case, the matrix representations of the
2-linear maps are straightforward to describe:

\begin{proposition}\label{eq:twisted-matrix-rep}
Given a 1-morphism:
\begin{equation}
  (X,\alpha,s,t): (A,\theta_A) \ra (B,\theta_B)
\end{equation}
the 2-linear map $\FV^{U(1)}(X,\alpha,s,t)$ has matrix representation whose
  components are:
\begin{equation}\label{eq:twisted-2lin-components}
  \FV^{U(1)}(X,\alpha,s,t)_{\rho,\phi} = Hom_{Rep (\mathbb{C}^{s^{\ast}\theta_A} X)}((s^{\ast}\rho) , ((M_{\alpha})^{\ast} t^{\ast}\phi))
\end{equation}
\end{proposition}
\begin{proof}
  As in the untwisted case, this uses Frobenius reciprocity, in this
  case for representations of algebras.  The operation of pulling back
  a representation $\phi$ of $\mathbb{C}^{\theta_B}[B]$) is adjoint to
  the operation of pushing-forward a representation of
  $\mathbb{C}^{t^{\ast}\theta_B}[X]$.

  As in the untwisted case, the $\FV^{U(1)}(X,\alpha,s,t)$ takes a
  $\theta_A$-twisted representation of $A$, and pulls back to $X$,
  then pushes forward to $B$.  The difference is that we apply the map
  $M_{\alpha}$ between these steps.  By Frobenius reciprocity, we have
  the intertwiner space:
  \begin{eqnarray}
    &  Hom_{Rep (\mathbb{C}^{\theta_B}[B])}( t_{\ast} \circ M_{\alpha}^{\ast} \circ s^{\ast} \rho , \phi ) \\
\nonumber    = & Hom_{Rep (\mathbb{C}^{t^{\ast}\theta_B}[X])}( M_{\alpha}^{\ast} \circ s^{\ast} \rho , t^{\ast}\phi ) \\
\nonumber    = & Hom_{Rep (\mathbb{C}^{s^{\ast}\theta_A}[X])}( s^{\ast} \rho ,  (M_{\alpha})_{\ast} \circ t^{\ast}\phi )
  \end{eqnarray}
\end{proof}

So given irreducible twisted representations $\rho$ and $\phi$ of $A$
and $B$, (i.e. irreducible representations of the twisted groupoid
algebras), we find a component in a matrix for a 2-linear map as an
intertwiner space between the pulled-back representations
$s^{\ast}\rho$ and $t^{\ast}\phi$.  These are twisted representations of X, but a
priori we have that $s^{\ast}\rho$ is twisted by the cocycle
$s^{\ast}(\theta_A)$, and $t^{\ast}\phi$ is twisted by
$t^{\ast}(\theta_B)$.

We note that in principle, given two representations twisted by
different cocycles, we would take the vector space of global sections
of:
\begin{equation}\label{eq:twistedbundle}
\bar{(s^{\ast}\rho)} \otimes (t^{\ast}\phi)
\end{equation}
This is a vector bundle on the objects of $X$, and when the cocycles
coincide, it corresponds to the usual hom space.  The ``bar'' means we
take the dual representation of $s^{\ast}\rho$, which is a
$(s^{\ast}\theta_A)^{-1}$-twisted rep of X, so the tensor product
(\ref{eq:twistedbundle}) is a $(s^{\ast}\theta_A)^{-1} \times
(t^{\ast}\theta_B)$-twisted representation of X.

This suggests that a further generalization of our $\uispgpd$ may be
possible in which the condition (\ref{eq:iicocycondition}) can be
weakened.  We might only require the $\alpha$ be a cochain, and that
$s^{\ast}\theta_A$ and $t^{\ast}\theta_B$ should differ by the
coboundary of a cochain $\alpha$.  This would ensure that $M_{\alpha}$
still induces an algebra isomorphism, but not necessarily an
automorphism.  If the condition were even weaker, the spaces of
sections (\ref{eq:twistedbundle}) would not correspond to intertwiner
spaces.  A similar generalization should be possible for condition
(\ref{eq:icocycondition}), so that the coboundary of $\beta_Y$ gives
the difference between pullbacks of $\alpha_1$ and $\alpha_2$.

This generalization, however, is not necessary for our construction of
this ETQFT, so we will not consider it further here.

\subsection{Twisted ETQFT}

Finally, our main result asserts that the DW model can
be understood as factorized into the classical field theory and the
2-linearization ``quantization functor'' we have just defined.

The theory itself, as described originally by Dijkgraaf and Witten
\cite{DW}, is given in more explicit detail by Freed and Quinn
\cite{freed-quinn}, particularly in the situation of manifolds with
boundary, which is the case where an ETQFT is most appropriate.  This
is the description to which we will refer here when speaking of the DW
model.  In particular, much of the description we use is in
(\cite{freed-quinn}, Sec. 4) which describes its construction as a
modular functor.

We will describe how it is derived from the 2-functor we
have given as our ETQFT.  

\begin{theorem}\label{thm:maintheorem}
  Given a finite gauge group $G$ and 3-cocycle $\omega \in
  Z^3(BG,U(1))$, the symmetric monoidal 2-functor
  \begin{equation}
    Z_G^{\omega} =  \FV^{U(1)} \circ \fc{-}^{\omega} : \catname{3Cob_2} \ra \iiV
  \end{equation}
  reproduces the DW model with twisting cocycle
  $\omega$.
\end{theorem}
\begin{proof} First, we note that the DW model as described in
  \cite{freed-quinn} assigns a Hilbert space $E(Y)$ to each 2D
  manifold $Y$ with boundary.  We regard think of this manifold as a
  2D cobordism between 1D boundary components, and describe the
  correspondence between $E(Y)$ and the 2-linear map assigned by the
  $Z_G$ given by our construction.  Horizontal composition of 1- and
  2-morphisms corresponds to the trace over a tensor product of the
  data associated to a boundary in which incoming and outgoing
  boundary components are distinguished by orientation.

  So, to the 1D boundary in such a case, the DW model assigns a
  collection of labels.  These are irreducible representations of
  certain algebras. Due to the monoidal structure, this reduces to the
  case of the algebra assigned to a circle, which is:
  \begin{equation}
    A^{\ast} = \bigoplus_{[T]} L^{\ast}_{[T]}
  \end{equation}
  This is a direct sum over $[T]$, the distinct conjugacy classes in
  $G$, which is to say, the isomorphism classes of objects of
  $\fc{S^1}$.  The algebras $L^{\ast}$ and connecting isomorphisms
  between them form a line bundle over the space of $[T]$, which is
  classified by a cohomology class given by the transgression of the
  form there called $\hat{\alpha}$, and here $\omega$ in
  $H^3(BG,\mathbb{R}/\mathbb{Z})$.  (This is \cite{freed-quinn}
  Proposition 3.14).

  So the cocycles $\beta_{\fc{M}}$ are as we expect for a
  3-dimensional cobordism $M$.  We have summarized these functorial
  properties by the observation of \ref{sec:hqftcocycle} that the
  assignment of cocycles is an HQFT.

  Now the algebra structure of $L^{\ast}_{[T]}$ is such that its unit
  vectors form a central extension of the centralizer of $[T]$ (that
  is, $Aut([T])$ in the sense of the groupoid $\fc{S^1}$).  The
  central extension is classified by the cocycle just mentioned.  In
  our terminology, this says precisely that $L^{\ast}_{[T]}$ itself is
  the summand of $\mathbb{C}^{\tau_{S^1}(\omega)}[\fc{s^1}]$
  associated to $[T]$.  Thus,
  \begin{equation}
    A^{\ast} \cong \mathbb{C}^{\tau_{S^1}(\omega)}[\fc{S^1}]
  \end{equation}
  But our $Z_G^{\omega}$ assigns the circle the representation
  category of this algebra, which recovers the label set assigned by
  the DW model.

  Next, the DW model assigns a Hilbert space to each manifold with
  boundary $Y$ (in the following we use the notation of
  \cite{freed-quinn}, Section 3).  We will understand this to be a
  cobordism relating its boundary components.  Thus, this Hilbert
  space is to be understood as a 2-linear map.  As a Hilbert space,
  the $E(Y)$ are given as:
  \begin{equation}
    E(Y) = L^2(\overline{\mathcal{C}'_Y},\overline{\mathcal{L}_Y})
  \end{equation}
  That is, it is the space of (square-integrable, which condition is
  vacuous in the finite case) sections of a certain line bundle
  $\mathcal{L}_Y$ over the space $\mathcal{C}'_Y$ of flat connections
  (i.e. bundles with flat connections) on $Y$, which we would describe
  as the space of objects of $\fc{Y}$.  This bundle assigns a
  1-dimensional space to each such object, and to each homotopy of the
  classifying maps of these flat bundles (i.e. to each morphism $f$ of
  $\fc{Y}$) an isomorphism of these lines, given by (3.4) of
  \cite{freed-quinn}.  This isomorphism incorporates a factor which
  comes from an integral of $\hat{\alpha}$, or in our terms $\omega$.
  This factor is just the value of $\tau_Y(\omega)$ on $f$.

  A decomposition of the space of sections $E(Y)$ as a direct sum of
  components is given in (\cite{freed-quinn}, sec. 4) as:
  \begin{equation}
    E(Y) \cong \bigoplus_{\lambda} E(Y,\lambda) \otimes E_{\lambda}
  \end{equation}
  Here, the $\lambda$ run over all labels for the boundary: this is a
  product of labels $\lambda = (\lambda_i)_i$ over all boundary
  components $(\partial Y)_i$.  The representations $E_{\lambda}$
  associated to the whole boundary are therefore of the form
  $\otimes_i E_{\lambda_i}$. By the duality of $Hom$ and $\otimes$,
  these are isomorphic to the intertwiner spaces given in
  (\ref{eq:twisted-2lin-components}).  The $E(Y,\lambda)$ give the
  multiplicities of these representations.

  The above decomposition amounts to treating $E(Y)$ as a module for
  $A^{\star}$ for the algebra associated to $\partial Y$, which acts
  on the $E_{\lambda}$, or rather as a bimodule for the algebras
  $A^{\star}$ for the source and target objects (taking the conjugate
  algebra when changing orientation, hence turning a left action into
  a right action).  Frobenius reciprocity then ensures that taking a
  tensor product with this bimodule will act as multiplying by the
  matrix (\ref{eq:twisted-2lin-components}). This gives an
  interpretation of $E(Y)$ as the 2-linear map $Z_G^{\omega}(Y)$.

  The DW model then assigns a map between these Hilbert space $E(Y)$
  for each cobordism between manifolds $Y$ and $Y'$.  We further note
  that the inner product on this space, as a space of sections, is
  twisted by the cocycle $\alpha$, which is accomplished precisely by
  the inclusion of the map $M_{\alpha}$ in our construction of the
  2-linear map $Z_G^{\omega}(Y)$.

  Finally we check that $Z_G^{\omega} = \FV^{U(1)} \circ
  \fc{-}^{\omega}$ gives the data of the twisted DW model for
  2-morphisms $M$ of $\catname{3Cob_2}$, which are understood as
  3-dimensional cobordisms of manifolds with boundary.  In
  \cite{freed-quinn} they are described as manifolds with corners.

  Part of this proof is substantially the same as that of Theorem
  \ref{thm:DWETQFT}, which shows in the untwisted case with empty
  boundary that our formula reproduces the (unnumbered) formula
  directly following (\cite{freed-quinn}, 5.14).  That formula uses
  the ``mass'' of a connection on (there described as a
  ``representation'' of $\pi_1(M)$ into the gauge group - though they
  denote the manifold by $Y$), which is just the groupoid cardinality
  $\frac{1}{Aut([A])}$ for a class $[A]$ of connections (denoted there
  by $\gamma$), as in our formula.  This gives the measure used in the
  integrals over the space of connections, as we expect.

  Finally, the explicit calculations of amplitudes in
  \cite{freed-quinn} are generally contractions of the 2-linear maps
  we obtain.  Moreover, they are converted to amplitudes from linear
  operators between representation spaces by converting
  representations to characters, taking the trace.  Thus, since the
  trace of the identity for a representation $\rho$ is $dim(\rho)$,
  the formulae there contain factors of $dim(E_{\lambda})$, where
  $E_{\lambda}$ is the representation space for a representation on
  the whole boundary (for us, the tensor product of the
  representations determining a given component of the natural
  transformation).  So finally the computations of amplitudes such as
  (\cite{freed-quinn}, 5.4) for the torus are precisely the result of
  applying this procedure to the natural transformations from
  $\FV^{U(1)}$.

  We conclude that the DW model for manifolds with
  corners as presented in \cite{freed-quinn} can be recovered from our
  $Z_G^{\omega}$.
\end{proof}

\section{Conclusion}

In this paper, our goal has been to give a concrete description of the
quantization functor which plays a role in the DW model.
This is consistent with the program of Freed-Hopkins-Lurie-Teleman
\cite{FHLT}, in which topological quantum field theories are described
in terms of a factorization into two parts.  The first part, the
classical field theory, takes values in groupoids.  The second part
assigns algebraic data to groupoids - in particular, $k$-vector
spaces, or indeed $k$-algebras in an appropriate sense, to the
groupoids associated to codimension-$k$ manifolds.

The point is that the groupoids represent the moduli space for the
field configurations of the classical theory.  As we have seen, the
full functor, to describe the DW model in its complete form, must
incorporate the effect of data from groupoid cohomology.

One purpose of studying the quantization functor separately is that we
hope to gain some understanding of the nature of the quantization
process.  Quantization is well-studied in the situation of a process
(in good situations, a functor) taking classical configuration spaces
to quantum Hilbert spaces.  The higher-categorical $k$-vector spaces
are less commonly used in the physical context and ETQFT gives a
sufficiently simple, yet nontrivial, setting in which to study this
aspect of quantization.  What our functors $\FV$ and its twisted
version $\FV^{U(1)}$ illustrate is that this process can be described
in terms of a simple, quite universal process once we understand the
category $\Span(\Gpd)$, or its twisted version $\uispgpd$.

In particular $\FV^{U(1)}$ is an extension of the very natural
"2-linearization" process $\FV$, which is entirely canonical.
Groupoids are taken to their representation categories.  The morphisms
(spans) are taken to 2-linear maps constructed naturally from
induction and restriction functors.  The 2-morphisms (spans of spans)
are taken to natural transformations constructed naturally from the
unit and counit for the adjunctions between these functors.  This is
an entirely canonical process generalizing the straightforward
``pull'' and ``push'' of functions through spans of sets which gives
(natural number valued) matrix multiplication.  Thus, the quantization
functor is simply giving a canonical representation of $\Span(\Gpd)$,
which is then in some sense the fundamental setting.

One important fact in the untwisted case is that, if we take the
representation category $\FV(A)$ to be concrete, with its natural
``underlying vector space'' functor into $\V$, we have a
Tannaka-Krein reconstruction theorem.  That is, this 2-vector space
(and the forgetful functor into $\V$) allows the groupoids (objects
of $\Span(\Gpd)$) to be reconstructed completely.  At the level of
morphisms, and particularly 2-morphisms, however, we do lose
information.  This is easy to see in the special case of
$Hom(\catname{1},\catname{1})$, where $\FV^{U(1)}$ restricts to give
groupoidification.  Here, spans of groupoids, as morphisms, are taken
to linear maps whose components detect only \textit{groupoid
cardinalities} of spans.  This does not determine the groupoids up to
isomorphism.  So in particular, the quantization functor is not
faithful, and forgets information about the classical category as part
of the ``sum over histories'' which defines the 2-morphisms.

The motivation for using $\Span(\Gpd)$ and its twisted extension is
how it reflects physically important aspects of the quantum field
theory.  The objects are groupoids because the moduli problem for
gauge theory, like many other geometric structures, has symmetries
which are not seen in a topological space of field configurations.
The quantization functor, our $\FV^{U(1)}$, is able to retain this
information about symmetry, since it assigns the representation
category to such groupoids.  This characterizes them up to Morita
equivalence.  In general, systems whose configuration spaces are
represented by Morita-equivalent groupoids are "physically
indistinguishable".

However, rather than working in the bicategory $\Gpd$, we expand it to
consider $\Span(\Gpd)$.  We have noted some work
(\cite{hoffnung-span}, \cite{kenney-pronk}) on such span categories
generally.  One important fact is that in general $\Span(\catname{C})$
for a (1-)category is a universal (bi)category containing a copy of
$\catname{C}$, for which every morphism has a (two-sided) adjoint.  In
the case $\catname{C}$ is a bicategory, such as $\Gpd$, our
construction also gives adjoints for 2-morphisms (in
\cite{hoffnung-span}, one gets a monoidal tricategory, which we have
made into a bicategory by taking 2-morphisms as mere equivalence
classes of spans of span maps).  This construction of ``adjoining
adjoints'' is somewhat analogous to localization, which forces
morphisms to be invertible.  Instead, we force morphisms to be
adjointable.  This is the key feature captured in $\Span(\Gpd)$, and
is also a key characteristic of the linear and 2-linear category.

The physical significance of adjointability is that if a morphism
describes a process by which a system evolves, its adjoint is the same
process with the reversed time-sense.  In the ETQFT case, the
cobordism category suggests that we should think of 2-morphisms as
``time evolution'' in this sense. The 1-morphisms then describe a
space with boundary as linking its boundary components, and the
adjoint simply reverses the sense of input and output boundary
components.

In the twisted case, we must expand this setting to $\uispgpd$, but
this behaves much like $\Span(\Gpd)$ except that the groupoids carry
extra cocycle information.  This information is the higher-categorical
extension of the Lagrangian functional, which is simply the 0-cocycle
associated to 2-morphisms.  This fits the approach of \cite{FHLT}, in
which the cocycle $\omega$ on $BG$, and the gerbe it classifies, is
taken to be the true physical setting for the action.  The
transgressions to the moduli spaces for connections on manifolds of
different dimensions are then particular manifestations of this action.

In subsequent work, it may be of interest to consider whether this
larger bicategory $\uispgpd$, or perhaps a weaker variation, has some
important universal properties analogous to those of $\Span(\Gpd)$.
For now, it is sufficient to observe that it is the natural target for
the classical field theory of the DW model, and likely other
interesting toy physics models relevant to TQFT.  A subsequent paper
by the author with Derek Wise will consider an analogous construction
with compact Lie groups and give an explicit construction of a
generalization of $\FV$ and $\FV^{U(1)}$ which applies in the infinite
setting.

\section{Acknowledgements}

The author would like to thank the referees of earlier versions for
important suggestions, as well as John Baez, Dan Christensen, Alex
Hoffnung, Thomas Nikolaus, Mike Stay, Jamie Vicary, and Derek Wise for
valuable discussions.

This work was partially financed by Portuguese funds via the Fundacao
para a Ciencia e a Tecnologia, through project number
PTDC/MAT/101503/2008, New Geometry and Topology.

\bibliography{ETQFT-DW}
\bibliographystyle{acm}

\end{document}